\documentclass[10pt,english]{amsart}
\usepackage{amsmath,amssymb,amsthm,amsfonts,graphicx,color}
\usepackage{amssymb,geometry}

\usepackage[colorlinks=true]{hyperref}

\makeatletter
\def\@currentlabel{2.1}\label{e:dispaa}
\def\@currentlabel{2.21}\label{e:dispau}
\def\@currentlabel{2.22}\label{e:dispav}
\def\@currentlabel{2.23}\label{e:dispaw}
\def\@currentlabel{2.24}\label{e:dispax}
\def\theequation{\thesection.\@arabic\c@equation}
\makeatother

\hypersetup{linkcolor=black,citecolor=black,filecolor=black,urlcolor=black} % black links, for printed output

\definecolor{dullmagenta}{rgb}{0.4,0,0.4}   % #660066
\definecolor{darkblue}{rgb}{0,0,0.4}

\newcommand{\R} {\mathbb R}
\newcommand{\cuad}{{\sqcap\kern-.68em\sqcup}}

\newcommand{\be}{\begin{equation}}
\newcommand{\ee}{\end{equation}}

\newcommand{\sech}{\mathop{\mbox{\normalfont sech}}\nolimits}
\newcommand{\supp}{\mathop{\mbox{\normalfont supp}}\nolimits}

\makeatletter
\newcommand*\rel@kern[1]{\kern#1\dimexpr\macc@kerna}
\newcommand*\widebar[1]{%
  \begingroup
  \def\mathaccent##1##2{%
    \rel@kern{0.8}%
    \overline{\rel@kern{-0.8}\macc@nucleus\rel@kern{0.2}}%
    \rel@kern{-0.2}%
  }%
  \macc@depth\@ne
  \let\math@bgroup\@empty \let\math@egroup\macc@set@skewchar
  \mathsurround\z@ \frozen@everymath{\mathgroup\macc@group\relax}%
  \macc@set@skewchar\relax
  \let\mathaccentV\macc@nested@a
  \macc@nested@a\relax111{#1}%
  \endgroup
}
\makeatother
\renewcommand{\theequation}{\thesection.\arabic{equation}}
 
 \newtheorem{lemma}{Lemma}[section]
\newtheorem{definition}{Definition}
\newtheorem{theorem}{Theorem}[section]
\newtheorem{proposition}{Proposition}[section]
\newtheorem{corollary}{Corollary}[section]
\newtheorem{remark}{Remark}[section]
\newcommand{\bremark}{\begin{remark} \em}
\newcommand{\eremark}{\end{remark} }

\begin{document}

\title[Maximal solution of the Liouville equation]
{Maximal solution of the Liouville equation in doubly connected domains}
%%%%%%%%%%%%%%%%%%%%%%%%%%%%%%%%%%%%%%%%%%%%%%%%%%%%%%%%%%%%%%%%%%%%%%
%
%\author{Micha{\l} Kowalczyk}
%\address{\noindent  M. Kowalczyk - Departamento de
%Ingenier\'{\i}a  Matem\'atica and CMM, Universidad de Chile,
%Casilla 170 Correo 3, Santiago,
%Chile.}
%\email{kowalczy@dim.uchile.cl}
%

%%%%%%%%%%%%%%%%%%%%%%%%%%%%%%%%%%%%%%%%%%%%%%%%%%%%%%%%%%%%%%%%%%%%%%
%%%%%%%%%%%%%%%%%%%%%%%%%%%%%%%%%%%%%%%%%%%%%%%%%%%%%%%%%%%%%%%%%%%%%%

\author{Micha{\l } Kowalczyk}
\address{Departamento de Ingenier\'{\i}a Matem\'atica and Centro
de Modelamiento Matem\'atico (UMI 2807 CNRS), Universidad de Chile, Casilla
170 Correo 3, Santiago, Chile.}
\email {kowalczy@dim.uchile.cl}
\thanks{M. Kowalczyk was partially supported by Chilean research grants Fondecyt 1130126 and 1170164 and Fondo Basal AFB170001 CMM-Chile. Part of this work was done during his visits at the University of Warsaw and Hiroshima University. A. Pistoia was partially supported by Sapienza   research grant  ``Nonlinear PDE's in geometry and physics''. G. Vaira was partially supported by GNAMPA research grant ``Esistenza e molteplicit\'a di soluzioni per alcuni problemi ellittici non lineari".}

\author{Angela Pistoia}\address{Dipartimento SBAI, Sapienza Universit\'a di Roma, via Antonio Scarpa 16, 00161, Roma, Italy.}
\email {angela.pistoia@uniroma1.it}

\author{Giusi Vaira}\address{Dipartimento di Matematica e Fisica,   Universit\'a della Campania ``L. Vanvitelli", viale Lincoln 5, 81100, Caserta, Italy.}
\email {giusi.vaira@unicampania.it}

%%%%%%%%%%%%%%%%%%%%%%%%%%%%%%%%%%%%%%%%%%%%%%%%%%%%%%%%%%%%%%%%%%%%%%
\begin{abstract}
In this paper we  consider the Liouville equation $\Delta u +\lambda^2 e^{\,u}=0$ with Dirichlet boundary conditions in a two dimensional, doubly connected domain $\Omega$.  We show that there exists a simple, closed curve $\gamma\subset \Omega$  such that for a sequence  $\lambda_n\to 0$ and a sequence of solutions $u_{n}$ it holds $\frac{u_{n}}{\log\frac{1}{\lambda_n}}\to H$, where $H$ is a harmonic function in $\Omega\setminus\gamma$ and 
$\frac{\lambda_n^2}{\log\frac{1}{\lambda_n}}\int_\Omega e^{\,u_n}\,dx\to 8\pi c_\Omega$, where $c_\Omega$ is a  constant depending on the conformal class of $\Omega$ only.   
\end{abstract}
%%%%%%%%%%%%%%%%%%%%%%%%%%%%%%%%%%%%%%%%%%%%%%%%%%%%%%%%%%%%%%%%%%%%%%
\date{}\maketitle

%\tableofcontents
%%%%%%%%%%%%%%%%%%%%%%%%%%%%%%%%%%%%%%%%%%%%%%%%%%%%%%%%%%%%%%%%%%%%%%
%%%%%%%%%%%%%%%%%%%%%%%%%%%%%%%%%%%%%%%%%%%%%%%%%%%%%%%%%%%%%%%%%%%%%%

\section{Introduction}
\setcounter{equation}{0}

 In this paper we prove existence  of new kind of solutions  for the Liouville equation:
 %\marginpar{\textcolor{red}{I substituted\\ $\in$ with $\subset$}}

\begin{equation}
\label{liu 1}
\begin{aligned}
\Delta u+\lambda^2 e^{\,u}&=0, \quad\mbox{in}\quad \Omega\subset \R^2,\\
u&=0, \quad\mbox{on}\quad  \partial\Omega.
\end{aligned}
\end{equation}
We  assume that  $\lambda>0$ is a small parameter and $\Omega$ is a bounded, smooth domain, which is doubly connected and such that the bounded component of $\R^2\setminus \Omega$ is not a point.

Considering the  existence of solutions to (\ref{liu 1}) we recall some well known facts. First, for all small $\lambda$  and for any bounded domain there exists the {minimal solution} $u_\lambda $, which has the property that
\begin{align*}
u_\lambda \to 0, \quad \mbox{as}\ \lambda \to 0.
\end{align*}
Second, there exist unbounded as $\lambda\to 0$ solutions related to the phenomenon of bubbling. In particular,  provided that $\Omega$ is not simply connected for any $k>1$ there exists the {\it k-bubble  solution} $u_\lambda^{(k)}$ which, as $\lambda\to 0$  blows up at $k$ isolated points $a_j\in \Omega$. 
Near each $a_j$  the local profile of the bubble is a scaling of 
\begin{equation}
\label{stndr bubl}
w(r)=\log\frac{8}{(1+r^2)^2}
\end{equation}
which is a radial 
%\marginpar{\textcolor{red}{added}}
solution of $\Delta w+e^{\,w}=0$ in $\R^2$, we call it  the standard bubble. 
When $x\approx a_j$ we have  $u_\lambda^{(k)}(x)\approx w(|x-a_j|/\lambda)-4\log\lambda$,   so  that 
\begin{equation}
\label{mass 1}
\lambda^2 \int_\Omega e^{\,u^{(k)}_\lambda}\to 8\pi k, \quad \mbox{as}\ \lambda \to 0,
\end{equation}
while for the minimal solution this last limit is $0$. 

In this paper  we will prove the existence of  solutions to (\ref{liu 1}) such that 
\begin{align}
\label{liu 2}
\lambda^2 \int_\Omega e^{\,u_\lambda}\,dx\to\infty, \quad\mbox{as}\ \lambda\to 0.
\end{align}
We will call $u_\lambda$ satisfying (\ref{liu 2}) the  {maximal solution}.
In view of the above discussion it is evident from (\ref{liu 2}) that the maximal solution can not blow up  only on  a finite set of points,  and we will demonstrate that in fact its  blow up set is the whole  $\Omega$ and it has a curve of bubbles of different type than the standard bubble. 

To state our result we will recall the notion of the harmonic measure of a closed curve $\gamma\subset\Omega$.
\begin{definition}\label{liu def 1}
Let $\gamma$ be a smooth, simple closed curve in $\Omega$. A function $H_\gamma\in {C}^{2}(\Omega\setminus \gamma)\cap {C}^0(\Omega) $ is called the harmonic measure of $\gamma$ in $\Omega$ if the following holds:
\begin{equation}
\begin{aligned}
\label{liu 3}
\Delta H_\gamma&=0,\quad \mbox{in}\ \Omega\setminus \gamma,\\
H_\gamma&=0, \quad \mbox{on}\ \partial\Omega,\\
H_\gamma&=1, \quad\mbox{on}\ \gamma.
\end{aligned}
\end{equation}
\end{definition}
%\marginpar{\textcolor{red}{$H_\gamma$ is the capacity of the curve $\gamma,$ isn't it? Capacity would be related with the integral of $\partial_n H_\gamma$ or the its Dirichlet functional. It is a number. }}
Since $\gamma$ is a simple closed curve it divides the plane, and consequently $\Omega$, into two disjoint components.
With fixed orientation on $\gamma$   one of these components can be called interior with respect to $\gamma$. We will  denote it by $\Omega^+$. The exterior component  will be denoted by $\Omega^-$.  We will also set:
\begin{align*}
H^\pm_\gamma=H_\gamma\left|_{\Omega^\pm}\right. .
\end{align*}
Functions $H^\pm_\gamma$ are harmonic in their respective domains $\Omega^\pm$ and they satisfy homogeneous Dirichlet boundary conditions on $\partial\Omega^\pm\setminus\gamma$. Finally for future purpose by $n$ we will denote the unit normal vector on $\gamma$. With the orientation chosen as above $n$ is the exterior  unit normal of $\Omega^-$ on $\gamma$, and  the interior unit normal of $\Omega^+$ on $\gamma$.
Before stating our main result we need the following:
\begin{lemma}\label{lemma free}
Let $\Omega\subset \R^2$ be a bounded,  smooth, doubly connected set  such that the bounded component of $\R^2\setminus \Omega$  is not a point. There exists a simple, closed and smooth curve $\gamma\subset \Omega$ such that its harmonic measure satisfies
\begin{align}
\label{liu 4}
\partial_n H_\gamma^++\partial_n H^-_\gamma=0, \quad \mbox{on}\ \gamma.
\end{align}
\end{lemma}

By the generalization of the Riemann mapping theorem for multiply connected domains there exists a  holomorphic, bijective map $\psi\colon\Omega\to B_{R_1}\setminus B_{R_2}$ with some $R_1>R_2>0$. Later on (section \ref{sec free}) we will show that $\psi(\gamma)=\partial B_{R}$, $R=\sqrt{R_1 R_2}$.  
%\marginpar{\textcolor{red}{$\psi(\gamma)=\partial B_{R}$}}
Given this our  main result is the following:
\begin{theorem}\label{theorem liu}
Under the hypothesis and with the notation of Lemma \ref{lemma free} there exist a sequence $\lambda_n\to 0$,  and a sequence of  maximal solutions $u_n$   of the Liouville problem (\ref{liu 1}) with the following properties:
\begin{itemize}
\item[(i)]
It holds
\begin{align}
\label{liu 5}
\frac{u_{n}}{2\log\frac{1}{\lambda_n}}\longrightarrow H^\pm_\gamma,\quad \mbox{as}\ \lambda_n\to 0,
\end{align}
%\marginpar{\textcolor{red}{maybe we could write  $ -\log 2\lambda_n^2 $\\ instead of \\ ${2\log\frac{1}{2\lambda_n}+\log 2}$}}

over the compact subsets of $\Omega^\pm$.
\item[(ii)]
We have
\begin{align}
\label{liu 6}
\frac{\lambda_n^2}{2\log\frac{1}{\lambda_n}}\int_\Omega e^{\,u_{n}}\,dx\longrightarrow \frac{4\pi}{\log{\sqrt{\frac{R_1}{R_2}}}}\quad \mbox{as}\ \lambda_n\to 0.
% \int_0^{|\gamma|} \partial_n H^{-}_\gamma\, ds=-\int_0^{|\gamma|} \partial_n H^{+}_\gamma\, ds, \quad \mbox{as}\ \lambda_n\to 0,
\end{align}
%where $|\gamma|$ is the length of $\gamma$ and $s$ is its arc length parameter.
\end{itemize}
\end{theorem}
By analogy with the expression in (\ref{mass 1}) we interpret the right hand side of  (\ref{liu 6}) as the mass of the blow up  of the maximal solution. Note that it depends solely  on the conformal class of $\Omega$. 

In our theorem we only claim the existence of a sequence of maximal solutions but  our results can be improved somewhat.  The sequence  $\{\lambda_n\}_{n=1, \dots}$ can be replaced by an open set $\Lambda\subset (0,1)$  of $\lambda$ such that $0\in \bar \Lambda$ and the result is true for any sequence $\{\lambda_n\}\subset \Lambda$ converging to $0$, c.f. Corollary \ref{cor fix 2}. On the other hand $\Lambda$ can not be replaced by an open interval. This is because our problem is resonant or in other words the Morse index of the maximal  solution becomes unbounded  as $\lambda \to 0$.

We will now discuss some previous  results related to Theorem \ref{theorem liu}. In \cite{nagasaki1990} Nagasaki and Suzuki proved that for a sequence  of solutions $u_n$ of (\ref{liu 1}) one of the following  holds as $\lambda_n\to 0$ 
\begin{itemize}
\item[(i)] $\|u_{n}\|_{L^\infty(\Omega)}\to 0$;
\item[(ii)] Solutions $u_{n}$ form $k$  bubbles i.e. they blow up at the set of isolated points in $\Omega$.
\item[(iii)] Blow up occurs in the whole $\Omega$ i.e. $u_{n}(x)\to \infty$ for all $x$.
\end{itemize}
From \cite{NAGASAKI1990144} we know that at least in the case of the annulus all three alternatives may occur. In particular solutions satisfying (i) or (iii) are radial and those satisfying (ii) have $k$ fold symmetry-this simplifies greatly the matter. For general domains by minimizing in $H^1_0(\Omega)$
the energy
%\marginpar{\textcolor{red}{added}} 
\[
\int_\Omega \frac{1}{2} |\nabla u|^2-\lambda^2\int_\Omega e^{u},
\]
which is possible when $\lambda$ is small, one can get the minimal solution described in (i). Constructing unbounded, bubbling solutions is more involved. Summarizing the results in \cite{Baraket1997, MR2157850} (for a similar result for the mean field equation see \cite{Esposito2005}) we know that for any multiply connected domain there exists a solution with $k\in \mathbb N$   bubbles. Our result completes the picture showing that at least in the  case of general doubly connected domains  all three alternatives hold. What is more we also describe the way that the whole domain blow up actually happens. First, we find the curve of concentration of the blow up in terms of the free boundary problem (\ref{liu 3})--(\ref{liu 4}), second we find the exact form of the line bubble in terms of the one dimensional solution of the Liouville equation (see section \ref{subsec 1} below), third we describe how this local behavior is mediated with the far field approximation of the solution given by the scaled  harmonic measures $H^\pm_\gamma$. 

The problem of determining  the curve $\gamma$ is interesting in its own right. The case of doubly connected domains is relatively simple because of the conformal equivalence between $\Omega$ and an annulus. For general multiply connected domains solving (\ref{liu 3})--(\ref{liu 4}) appears to be more complicated but it can be handled by a simply trick (see \cite{kprv2018}).   In principle our method in the proof of Theorem \ref{theorem liu} applies in the more general case  but since it would require some additional, hard to verify assumption about $\gamma$  needed to solve the matching problem (Proposition \ref{prop inner 1} to follow) we  do not pursue it here.

Liouville's equation (\ref{liu 1}) belongs to a larger family of problems, one of them is the mean field model
\[
\Delta_g u+\rho\left (\frac{V(x) e^{\,u}}{\int_M V e^{\,u}\,d\mu}-1\right)=0,
\]
on a compact two dimensional,  Riemannian manifold $(M,g)$ assuming $V>0$. This equation appears in statistical mechanics \cite{Caglioti1992, Caglioti1995, doi:10.1002/cpa.3160460103} and Chern-Simmons-Higgs theory \cite{PhysRevLett.64.2234,PhysRevLett.64.2230} and its  most complete existence theory was developed by Chen and Lin \cite{doi:10.1002/cpa.3014,doi:10.1002/cpa.10107}, see also \cite{Li1999,DING1999653,Struwe1998, MR2483132,MR2671126, MR2409366,MR2456884,MR3263503}. In very general terms, given suitable  assumptions on $M$,  these results show the existence of bubbling solutions as the parameter $\rho\to \infty$ with the number of bubbles increasing to infinity as 
$\rho$ crosses the values $\rho_m=8\pi m$, $m\in \mathbb N$.    Another related problem we should mention is to find a conformally equivalent metric with  prescribed Gaussian curvature  on a  given Riemannian manifold $(M, g_0)$.
When $M=S^2$ and $g_0$ is the standard metric on the sphere this is known as the Nirenberg problem.  
This problem was  studied by Kazdan and Warner \cite{10.2307/1970993} who gave sufficient and necessary conditions for the existence of solutions in certain cases  (see also \cite{chang1987,chang2004non} and the references therein).  Apparent similarity between these problems and (\ref{liu 1}) is  due to the exponential nonlinearity but much deeper relation is expressed by the theorem of Br\'ezis and Merle  \cite{doi:10.1080/03605309108820797} which deals with the problem of classifying all possible limit points of sequences of solutions to 
\[
\Delta u_n+V_n(x) e^{\,u_n}=0, 
\]
in a domain  $D\subset \R^2$,  where {\it a priori} it is assumed $0\leq V_n\leq C$ and also that
$$
\int_D e^{\,u_n}< C ,
$$
for  a constant $C$
%\marginpar{\textcolor{red}{use the same constant}} 
(for  recent related results see  \cite{OHTSUKA2007419, LIN2016493}). 
Under these conditions one of the following holds 
\begin{itemize}
\item[(i)] $u_n$ is bounded in $L^\infty_{\mathrm{loc}}$;
\item[(ii)] $u_n\to -\infty$ on compacts;
\item[(iii)] $u_n$ blows up as standard bubbles along  a finite set of isolated points;
\end{itemize}
(see also \cite{MR1322618}).  This result and \cite{nagasaki1990} are clearly counterparts-the difference between them comes mainly from lack of boundary conditions and the finite mass assumption in \cite{doi:10.1080/03605309108820797}. From this perspective it is tempting to conjecture that the maximal solution should  exist at least  in some situations  for the above mentioned problems and that they may play an important role in understanding  globally the set of solutions. The present  paper is a first step in this direction in the context of (\ref{liu 1}). In \cite{kprv2018} we derived formally the free boundary problem associated with the maximal solutions for the mean field model which also supports the possibility of  its  existence in more general settings. For a recent result in this direction for the prescribed Gaussian curvature problem we refer to \cite{lsmr2018}.

Our work was in part inspired by \cite{gladiali2007} where (\ref{liu 6}) of our theorem was proven for the annuli. Another closely related results are contained in \cite{DELPINO20163414,dpw2006, pv}  where solutions of the stationary Keller-Segel system concentrating on the boundary of the domain were constructed.  Finally we mention \cite{Bonheure2017, MR3625083,2017arXiv170910471B} (see also \cite{BONHEURE2016455}) where solutions with unbounded mass to the same problem were found  in the radially symmetric setting.

The proof of Theorem \ref{theorem liu} relies  on a  fixed point argument that is designed to  handle at the same time the problem of matching  the inner and the outer expansion and of large inverse of the inner linear operator. 
To deal with  the matching  problem we need to adjust the solution near $\gamma$ (the inner solution) by  linearly growing terms in order to match it with the outer solution. Here our approach shares some similarities with  the end-to-end constructions by Traizet \cite{MR1427131}, Jleli and Pacard \cite{MR2194146}, Ratzkin \cite{MR1986894} and Kowalczyk, Liu, Pacard and Wei \cite{Kowalczyk:2014sf}.  Apart  from that we need to adjust the position of $\gamma$ in order to make sure that a set of orthogonality conditions is satisfied to guarantee that   the  inner linear operator is invertible-this is the Lyapunov-Schmidt reduction step. The novelty  is the intricate combination of the end-to-end construction and the Lyapunov-Schmidt reduction.

This paper is organized as follows: in section \ref{sec free} we introduce the approximate maximal solution. In the following two sections we develop the  linear theory to deal with the outer and the inner problem respectively. In the final section we carry out the fixed point argument.

We would like to thank Christos Sourdis and Piotr Rybka for useful  discussions during the preparation of this paper.

\section{The approximate solution}\label{sec approx sol}
\setcounter{equation}{0}

\subsection{Existence of the free boundary}\label{sec free}

In this section we will prove Lemma \ref{lemma free}.  For the rest of this paper $\Omega$ will be a doubly connected, bounded subset of $\R^2$ as described in the hypothesis of the Lemma. Under these assumptions it is known that $\Omega$ is conformally equivalent to an annulus (see Theorem 4.2.1 in \cite{krantz_book1}). Thus we have a holomorphic, bijective map $\psi\colon\Omega\to B_{R_1}\setminus B_{R_2}$ with some $R_1>R_2>0$. To show Lemma \ref{lemma free} we first solve the problem of finding the curve $\gamma$ in the annulus $A=B_{R_1}\setminus B_{R_2}$ and then pull it back to $\Omega$ using $\psi$. As a candidate for $\gamma$ we take $C_R=\{|x|=R\}$, where $R$ will be adjusted to satisfy the free boundary problem.  We will denote $A^+=\{R_1>|x|>R\}$ and $A^-=\{R>|x|>R_2\}$. The functions $H^\pm_{C_R}$ should satisfy the following set of conditions
\begin{equation}\label{an free}
\begin{aligned}
\Delta H^\pm_{C_R}&=0, \quad \mbox{in}\quad  A^\pm,\\
H^\pm_{C_R}&=0, \quad \mbox{on}\quad  \partial A^\pm\cap \partial A,\\
H^\pm_{C_R}&=1, \quad \mbox{on}\quad C_R,\\
\partial_r H^+_{C_R}+\partial_r H^-_{C_R} &=0, \quad \mbox{on}\quad C_R.
\end{aligned}
\end{equation}
It is rather easy to see that we should have
\[
H^\pm_{C_R}=a^\pm+b^\pm \log r,
\]
and that all conditions in (\ref{an free}) will be satisfied when $R=\sqrt{R_1 R_2}$ and
\begin{equation}
\label{abplus}
\begin{aligned}
a^+=-\frac{\log R_1}{\log\left(\sqrt{\frac{R_2}{R_1}}\right)}, &\quad a^-=-\frac{\log R_2}{\log\left(\sqrt{\frac{R_1}{R_2}}\right)},\\
b^+=\frac{1}{\log\left(\sqrt{\frac{R_2}{R_1}}\right)}, &\quad b^-=\frac{1}{\log\left(\sqrt{\frac{R_1}{R_2}}\right)}.
\end{aligned}
\end{equation}
Note that $b^-+b^+=0$ hence we have
\[
\partial_r H^-_{C_R}+\partial_r H^+_{C_R}=\frac{b^-}{\sqrt{R_1 R_2}}+\frac{b^+}{\sqrt{R_1 R_2}}=0.
\]
We let $\gamma=\psi^{-1}(C_R)$ and 
\[
H^\pm_{\gamma}=H^\pm_{C_R}\circ \psi.
\]
Since $\psi$ is a conformal map it is evident that $H^\pm_\gamma$ satisfies the assertions of Lemma \ref{lemma free}. Observe that by definition we have
\[
\partial_n H^-_\gamma=|\partial_n\psi|\partial_r H^-_{C_R}\circ \psi=-|\partial_n\psi|\partial_r H^+_{C_R}\circ\psi=-\partial_n H^+_\gamma,
\]
hence along $\gamma$ 
\begin{equation}
\label{hplusminus}
\partial_n H^-_\gamma=|\partial_n\psi| \frac{b^-}{\sqrt{R_1 R_2}}, \qquad \partial_n H^+_\gamma=|\partial_n\psi| \frac{b^+}{\sqrt{R_1 R_2}}.
\end{equation}

\subsection{The interior and exterior approximations and their matching condition}\label{subsec 1} 

\subsubsection{Scaling and local coordinates near $\gamma$}
First we introduce some scaling functions of the small parameter $\lambda$. They  will be used frequently in  the rest of  this paper. We set
\begin{align}
\beta &=2\log\frac{1}{a_0\lambda}+b_0,\quad a_0=2, b_0=\log 2, \label{beta}\\
\mu_\lambda&=-\frac{(\beta+2\log\beta)}{a_0\lambda}\partial_n H_\gamma^+=\frac{(\beta+2\log\beta)}{a_0\lambda} \partial_n H_\gamma^- =|\partial_n\psi|\frac{(\beta+2\log\beta)}{a_0\lambda} \frac{b^-}{\sqrt{R_1 R_2}}>0. \label{mu}
\end{align}
Note that $\mu_\lambda$ is a function on $\gamma$. 

Next we define  the Fermi  coordinates of the curve $\gamma$. We will denote the arc length  parametrization of $\gamma$ by $s$ and let $t(x)=\mathrm{dist}(\gamma, x)$ to be the signed distance to $\gamma$ chosen in agreement with its orientation so that 
\[
x=\gamma(s)+t n(s).
\]
For every point $x$ sufficiently close to $\gamma$,  the map
$x\longmapsto (s, t)$
is a diffeomorphism. We will denote this diffeomorphism by $X_\gamma$, so that $X_\gamma(x)=(s,t)$. In what follows we will often express functions globally defined in 
$\Omega$ or $\Omega^\pm$ in terms of the Fermi coordinates keeping in mind that these expressions are correct only when $|\mathrm{dist}(x,\gamma)|<\delta$ with some $\delta>0$ small. 
%To streamline the notation we will use symbols $u, v, w$ etc. to denote globally defined functions and let 
%\[
%{\tt u}= u\circ X^{-1}_\gamma,\quad  {\tt v}=v\circ X^{-1}_\gamma, \quad {\tt w}= w\circ X^{-1}_\gamma \quad {\mbox {etc.}}
%\] 
%to denote these functions expressed in the Fermi coordinates.  

\subsubsection{The initial approximation of the solution}

Let us consider  the following ODE:
\begin{align}
\label{1d profile}
\begin{aligned}
&u''+e^{\,u}=0, \quad \mbox{in}\ \R,\\
& u'(0)=0.
\end{aligned}
\end{align}
This problem has an explicit solution
\begin{align}\label{u}
U(t)=\log\left({2\sech^2 t}\right),
\end{align}
whose asymptotic behavior is given by:
\begin{align}\label{u-asy}
U(t)=-a_0|t|+b_0 +\mathcal{O}(e^{\,-a_0|t|}), \quad |t|\to \infty, \quad\mbox{where}\  a_0=2\ \hbox{and}\ b_0=\log 2.
\end{align}
Considering the equation (\ref{liu 1}) we observe that if $v$ is a solution of  $\Delta v+e^{\,v}=0$
then $u(x)=v(\lambda\mu x)+2\log\mu$ is, for any constant  $\mu>0$, a solution of $\Delta u +\lambda^2 e^{\,u}=0$. 
Motivated by this  we define the first inner approximation $v_0$  of the solution by:
\begin{equation}
\label{uzero 2}
v_0(x)=U\left(\lambda\mu_\lambda t\right)+2\log\mu_\lambda, \quad x=X_{\gamma}^{-1}(s,t).
\end{equation}
This definition is tentative and it will need to be modified later on but it will suffice for now. Note that $v_0$ is  well defined as a function of 
$x\in\Omega\cap\{|\mathrm{dist}\, (x,\gamma)|<\delta\}$ with some $\delta>0$. 
By \eqref{u} and \eqref{u-asy} we deduce that 
\begin{align}
\label{uzero 3}
v_0 \circ X^{-1}_\gamma (s,t)=-a_0 \lambda\mu_\lambda |t|+b_0+2\log\mu_\lambda+\mathcal O\left(e^{-a_0\lambda\mu_\lambda|t|}\right).
\end{align}

We consider  the outer approximation.   In each of the components $\Omega^\pm$ of $\Omega\setminus \gamma$ we define the outer  approximation $w_0^\pm$ by:
\begin{align}
\label{ext 1}
\begin{aligned}
\Delta w_0^\pm&=0, \quad\mbox{in}\ \Omega^\pm, \\
w_0^\pm&=0, \quad \mbox{on}\ \partial\Omega^\pm\cap\partial \Omega.
\end{aligned}
\end{align}
Note that in this problem we are missing the boundary condition on $\gamma$ for each of the unknown functions. This boundary condition will be determined through a matching condition between the inner approximation $v_0$ and the outer approximation $w_0^\pm$. To find what  it is explicitly let us express the outer approximations in terms  of the Fermi coordinates of $\gamma$, and then expand them formally in terms of the variable $t$:
\begin{align*}
w_0^\pm\circ X^{-1}_\gamma (s,t)&=w_0^\pm \circ X^{-1}_\gamma (s, 0)+t\partial_t w_0^\pm \circ X^{-1}_\gamma(s,t)\mid_{t=0}+\dots\\
&=w_0^\pm (s, 0)+t \partial_n w_0^\pm \circ X^{-1}_\gamma(s,0)+\dots
\end{align*}
Next we let $\eta$ to be the inner  variable
\begin{equation}
\label{def eta 1}
\eta=\lambda\mu_\lambda t \Longrightarrow t=\frac{\eta}{\lambda\mu_\lambda}
\end{equation}
that is more convenient to express the matching conditions.   
To match the inner and outer approximations  the following identities should  hold:
\begin{align}
\label{match 1}
\begin{aligned}
 w_0^+\circ X^{-1}_\gamma(s,0)+\frac{\eta}{\lambda\mu_\lambda}\left(\partial_n w_0^+\circ X^{-1}_\gamma\right)(s,0)&=-a_0\eta+b_0+2\log\mu_\lambda,\\
w_0^- \circ X^{-1}_\gamma(s,0)+\frac{\eta}{\lambda\mu_\lambda}\left(\partial_n w_0^-\circ X^{-1}_\gamma\right)(s,0)&=a_0\eta+b_0+2\log\mu_\lambda.
\end{aligned}
\end{align}
This leads to the following conditions on $\gamma$
\begin{equation}
\label{match 1a}
 \begin{cases}
w_0^+=b_0+2\log\mu_\lambda,\\
\partial_n w_0^+ =-a_0\lambda\mu_\lambda,
\end{cases} \ \hbox{and}\quad
 \begin{cases}
w_0^-=b_0+2\log\mu_\lambda,\\
\partial_n w_0^-=a_0\lambda\mu_\lambda.
\end{cases}
 \end{equation}
These boundary conditions together with the boundary conditions on $\partial \Omega\cap \partial \Omega^\pm$ give an overdetermined, nonlinear  problem for $\mu_\lambda$, which seems to be rather difficult. To avoid this complication we note that  at this point we only need to satisfy  conditions (\ref{match 1a})  with certain precision. For now it suffices that with $\mu_\lambda$ defined in (\ref{mu}) the difference between the left and the right hand sides is of order $\mathcal O\left(\frac{\log\log\frac{1}{\lambda}}{\log \frac{1}{\lambda}}\right)$ in the matching of $w_0^\pm$ and $\mathcal O(1)$ in the matching of $\partial_n w_0^\pm$.  To accomplish this we set
\begin{equation}\label{wzero}
w_0^\pm=(\beta+2\log\beta) H_\gamma^\pm+\tilde H^\pm,
\end{equation}
where
\[
\begin{aligned}
\Delta \tilde H^\pm&=0, \qquad \mbox{in}\ \Omega^\pm,\\
\tilde H^\pm &=0, \qquad \mbox{on}\ \partial\Omega^\pm\cap\partial\Omega,\\
\tilde H^\pm&=2\log(\mp\partial_nH_\gamma^\pm), \qquad \mbox{on}\ \gamma.
\end{aligned}
\]
We check 
\begin{equation}
\label{match 1b}
\partial_n w_0^\pm\pm a_0\lambda\mu_\lambda=\partial_n \tilde H^\pm=\mathcal O(1)\ \hbox{on}\ \gamma,
\end{equation}
so that the matching conditions for the derivatives are satisfied as we wanted. We claim that  
\begin{equation}
\label{match 1c}
 w_0^\pm-b_0-2\log\mu_\lambda=\mathcal O\left(\frac{\log\log\frac{1}{\lambda}}{\log\frac{1}{\lambda}}\right)\ \hbox{on}\ \gamma,
\end{equation}
as needed.
Indeed
\[
\begin{aligned}
w_0^+-b_0-2\log\mu_\lambda&=(\beta+2\log \beta)H^+_\gamma+2\log(-\partial_n H^+_\gamma)-b_0-2\log\left(\frac{1}{a_0\lambda}\right)\\
&\qquad -2\log\beta-2\log(-\partial_n H^+_\gamma)-2\log\left(1+\frac{2\log\beta}{\beta}\right)\\
&=-2\log\left(1+\frac{2\log\beta}{\beta}\right)\\
&= \mathcal O\left(\frac{\log\beta}{\beta}\right)=\mathcal O\left(\frac{\log\log\frac{1}{\lambda}}{\log\frac{1}{\lambda}}\right) ,
\end{aligned}
\]
 (we have  used $H_\gamma^+=1$ on $\gamma$). Of course we compute similarly the error of $w_0^--b_0-2\log\mu_\lambda$.

\section{The matching problem}\label{general improvement}
\setcounter{equation}{0}
Our objective is to solve the problem (\ref{liu 1}) by a fixed point argument built around a function that looks like $v_0$ near $\gamma$ and like $w^\pm_0$ in $\Omega^\pm$. 
This argument will involve  linear equations corresponding to the  inner problem and  two outer problems that  should  be coupled together through some matching conditions. In this and the next section we will develop a suitable linear theory to deal with this.

Let us begin with  a simple ODE. Denoting 
\[
L_\eta=\partial_\eta^2+e^{\,U}
\]
we are to find a solution of 
\begin{equation}
\label{inner 2}
L_\eta v=g.
\end{equation}
The fundamental set $\mathcal K$ of $L_\eta$ is given by
%$$\mathrm{Ker}\, L_\eta=\left\{c_1\varphi_1+c_2\varphi_2\ :\ c_1,\ c_2\ \in\R\right\},$$
%where 
\begin{equation}
\label{def ker aaa}
\mathcal K=\mathrm{span}\,\{\varphi_1=\eta U'+2 \quad\hbox{and} \quad \varphi_2=U'\}.
\end{equation}
The Wronskian of these functions $W=4$ and  the general  solution to \eqref{inner 2} is given by
$$v =c_1\varphi_1+c_2\varphi_2-\frac{1}{4}\varphi_1(\eta)\int_{-\infty}^\eta{\varphi_2(\xi) g(\xi)}\,d\xi+ \frac{1}{4}\varphi_2(\eta)\int_{-\infty}^\eta{\varphi_1(\xi) g(\xi)}\,d\xi, $$
where $c_1$ and $c_2$ are constants.
 To account for the decay of  solutions  to \eqref{inner 2}  we  introduce the weighted Sobolev spaces  $H^\ell_\theta(\R)=e^{\,\theta|\eta|} H^\ell(\R)$ equipped with their natural weighted norms. 
% \[
%H^2_\theta (\R)=\left\{u\in  H^2(\R)\ :\ \int\limits_{\R}e^{\,2 \theta |\eta|} \left[(u^{\prime\prime})^2+(u^\prime)^2+u^2\right]d\eta<\infty\right\}
%\]
%equipped with the norm
%$$\|u\|_{H^2_\theta(\R)}=\left(\int\limits_{\R}e^{\,2 \theta |\eta|} \left[(u^{\prime\prime})^2+(u^\prime)^2+u^2\right]d\eta\right)^{1/2} $$
%and  
%\[
%L^2_\theta (\R)=\left\{u\in  L^2(\R)\ :\ \int\limits_{\R}e^{ \,2\theta |\eta|} u^2 d\eta<\infty\right\}
%\]
%equipped with the norm
%$$\|u\|_{L^2_\theta(\R)}=\left(\int\limits_{\R}e^{\,2\theta |\eta|}  u^2 d\eta\right)^{1/2}.$$ 

Given this we have:
\begin{lemma}\label{prop inner 1}
Assume that $g\in L^2_\theta(\R)$ satisfies the following orthogonality conditions 
\[
\int_\R g(\eta)\varphi_1(\eta)\,d\eta=0=\int_\R g(\eta)\varphi_2(\eta)\,d\eta.
\]
Then (\ref{inner 2}) has a unique, decaying solution $v\in H^2_{\theta'}(\R)$, $0<\theta'<\theta$ and 
\[
\|v\|_{H^2_{\theta'}(\R)}\leq C\|g\|_{L^2_\theta(\R)}.
\]
\end{lemma}

We turn our attention now to the kernel $\mathcal K$ of $L_\eta$. Note that any function $h\in \mathrm{span}\,\mathcal K$ behaves asymptotically as $|\eta|\to \pm\infty$ as an affine function 
\begin{equation}
\label{3.0}
h(\eta)=h_1^\pm \eta+h_2^\pm +\mathcal O\left(e^{\,-2|\eta|}\right).
\end{equation}
In what follows we will consider, more generally, the space $\mathcal K_\gamma$  of functions $h\colon\gamma\times \R\to \R$ of the form
\[
h(s, \eta)=h_1(s)\varphi_1(\eta)+h_2(s)\varphi_2(\eta). 
\]
We identify $\mathcal K_\gamma$ with a product of Sobolev spaces and we will use the norms 
\[
\|h\|_{H^\ell(\mathcal K_\gamma)}=\left(\|h_1\|^2_{H^\ell(\gamma)}+\|h_2\|^2_{H^\ell(\gamma)}\right)^{1/2}.
\]

Next we will describe the outer problem. We look for functions $w^\pm$ such that 
\begin{equation}
\label{outer -1}
\begin{aligned}
\Delta w^\pm&=g^\pm,\qquad  \mbox{in}\ \Omega^\pm,\\
w^\pm&=0, \qquad \mbox{on}\ \partial\Omega^\pm\cap\partial\Omega.
\end{aligned}
\end{equation}
To solve our problem we need to match the inner and the outer expansions on $\gamma$ and for this purpose we will use a function  $h\in \mathcal K_\gamma$ with the affine behavior (\ref{3.0}).  We impose the following conditions along $\gamma$:
\begin{equation}
\label{3.1}
\begin{aligned}
w^\pm&=h_1^\pm,\qquad \mbox{on}\ \gamma, \\
 \partial_n w^\pm&={\lambda\mu_\lambda}h_2^\pm,  \qquad \mbox{on}\ \gamma.
\end{aligned}
\end{equation}
We call  (\ref{outer -1})--(\ref{3.1})  the {\it matching problem}.  Note that (\ref{3.1}) implies that in general the matching  problem is overdetermined and its solution are the functions $w^\pm$ together with the function $h$.  

Our results are more conveniently expressed in terms of  the large parameter  
\begin{equation}
\label{def alpha}
\alpha= \frac{(\beta+2\log\beta)}{a_0\sqrt{R_1 R_2}\log\sqrt{\frac{R_1}{R_2}}},  \quad\mbox{where}\  \beta\ \mbox{is defined in}\ (\ref{beta}).
\end{equation}

%%%%%%%%%%%%%%%%%%%%%%%%%%%%%%%%%%%%%%%%%%%%%%%%%%%%%%%%%%%%%%%%%%%%%%%%%%%%%%%%%%%%%
%%%%%%%%%%%%%%%%%%%%%%%%%%%%%%%%%%%%%%%%%%%%%%%%%%%%%%%%%%%%%%%%%%%%%%%%%%%%%%%%%%%%%%%%%

\begin{proposition}\label{cru} 

Let $q=\frac{R_2}{R_1}\in (0,1)$, $R=\sqrt{R_1 R_2}$. There exits a positive  integer $N_{R, q}$ and   a diverging, positive  sequence  $\{\alpha_n\}_{n=N_{R, q}, \dots}$  such that for all  $\alpha>\max\{4R, \frac{2}{-R\log q}\}$ satisfying
\[
 \min_{n>N_{R,q}} |\alpha_n-\alpha|>\frac{1}{4 R},
\]
the matching problem (\ref{outer -1})--(\ref{3.1}) has a solution $w^\pm\in H^2(\Omega^\pm)$, $h\in H^1(\mathcal K_\gamma)$ such that 
\begin{equation}
\label{match 556}
\|h\|_{H^1(\mathcal K_\gamma)}+\|w^+\|_{H^2(\Omega^+)}+\|w^-\|_{H^2(\Omega^+)}
\leq C \left(\|g^+\|_{L^2(\Omega^+)}+\|g^-\|_{L^2(\Omega^-)}\right).
\end{equation}

\end{proposition}
 \begin{proof}
We write
\[
h=h_1\varphi_1+ h_2\varphi_2.
\]
Taking into account that
$$\varphi_1(\eta)=2-2|\eta|+\mathcal O(e^{-2|\eta|}) \quad  \hbox{and}\quad  \varphi_2(\eta)= -2\frac\eta{|\eta|}+\mathcal O(e^{-2|\eta|}),\quad |\eta|\to \infty,$$
the matching  conditions  on $\gamma$ read:
\begin{align}
\label{outer 200}
\left\{\begin{aligned}
w^-&=2 h_1+2 h_2, \\
\partial_n w^-&=2\lambda\mu_\lambda h_1,\end{aligned}\right.
\qquad
\left\{\begin{aligned}
 w^+&=2h_1-2 h_2, \\
\partial_n  w^+&=-2\lambda\mu_\lambda  h_1.\end{aligned}\right.
\end{align}
We solve first  the Poisson equation
\begin{equation}
\label{3.2}
\begin{aligned}
\Delta \phi^\pm&=g^\pm, \qquad \mbox{in}\ \Omega^\pm,\\
\phi^\pm&=0, \qquad \mbox{on}\ \partial\Omega^\pm\cap\partial\Omega,\\
\partial_n\phi^\pm&=0, \qquad \mbox{on}\ \gamma.
\end{aligned}
\end{equation}
The solution of the matching  problem will be of the form $w^\pm=\tilde w^\pm+\phi^\pm$, where $\tilde w^\pm$ is  a  harmonic function in $\Omega^\pm$, satisfying the Dirichlet boundary conditions on $\partial\Omega^\pm\cap \partial \Omega$ and the  following matching conditions on $\gamma$
\begin{equation}
\label{3.3}
\left\{\begin{aligned}
\tilde w^-&=2h_1+2 h_2-\phi^-, \\
\partial_n \tilde w^-&=2\lambda\mu_\lambda  h_1,\end{aligned}\right.
 \qquad
\left\{\begin{aligned}
\tilde w^+&=2h_1-2 h_2-\phi^+\\
\partial_n \tilde w^+&=-2\lambda\mu_\lambda  h_1,\end{aligned}\right.
\end{equation}
Thus our problem is reduced to finding $\tilde w^\pm$ such  that $\Delta \tilde w^\pm=0$ in $\Omega^\pm$ and  (\ref{3.3}) are satisfied. 
We will use the conformal map $\psi$  in order to transfer the matching problem   from $\Omega$ to the annulus $B_{R_1}\setminus B_{R_2}$. Thus we will define the pullbacks of $w^\pm$ by $\psi$
\[
w^{*\pm} = \tilde w^\pm\circ \psi^{-1}
\] 
and in the same way we define the pullbacks $h^*_j$ by $\psi$ of  the functions  $h_j$  restricted to $\gamma$. We recall that $\psi(\gamma)=C_R=\{|x|=R\}$ with $R=\sqrt{R_1 R_2}$. An important observation is that 
\[
\partial_n \tilde w^\pm =|\partial_n \psi| \partial_r w^{*\pm}
\]
along $\gamma$. Taking this and (\ref{mu}) into account we see that the Neumann boundary conditions in (\ref{outer 200}) for  the pullbacks are
\begin{equation}
\label{neumann 1}
\begin{aligned}
\partial_r w^{*-} &=2\alpha h_{1}^*,\\
\partial_r w^{*+} &=-2\alpha h_{1}^*,
\end{aligned}
\end{equation}
along $C_R$, where $\alpha$ is defined in (\ref{def alpha}). 
By the way we note that 
\begin{equation}
\label{fa 1}
\lambda\mu_\lambda =\alpha |\partial_n\psi|,
\end{equation}
see the definition of $\mu_\lambda$ in (\ref{mu}).
Since  $\alpha= \mathcal O(\log\frac{1}{\lambda})$ is constant along $C_R$ we can solve the matching problem transferred to the annulus by the Fourier series method. 
The problem for $w^{*-}$ reads
\begin{equation}
\label{wmin sharp}
\begin{aligned}
\Delta w^{*-}&=0, \quad \mbox{in}\quad B_R\setminus B_{R_2}\\
w^{*-}&=0, \quad \mbox{on} \quad \partial B_{R_2}\\
w^{*-}&= 2h_{1}^*+2 h_{2}^*-\phi^{*-}, \quad \mbox{on}\quad C_R\\
\partial_r w^{*-} &=2\alpha h_{1}^*, \quad \mbox{on}\quad C_R,
\end{aligned}
\end{equation}
and the one for $w^{*+}$ becomes
\begin{equation}
\label{wplus sharp}
\begin{aligned}
\Delta w^{*+}&=0, \quad \mbox{in}\quad B_{R_1}\setminus B_{R}\\
w^{*+}&=0, \quad \mbox{on} \quad \partial B_{R_1}\\
w^{*+}&=2h_{1}^*-2 h_{2}^*-\phi^{*+}, \quad \mbox{on}\quad C_R\\
\partial_r w^{*+} &=-2\alpha h_{1}^*, \quad \mbox{on}\quad C_R.
\end{aligned}
\end{equation}
It is convenient to state the analog of Proposition \ref{cru} for problems (\ref{wmin sharp})--(\ref{wplus sharp}).
\begin{lemma}\label{lem cru}
Under the hypothesis of Proposition \ref{cru}  we have
\begin{equation}
\label{cru 1}
\|h_1^*\|_{H^1(C_R)}+\|h_2^*\|_{H^1(C_R)}+\|w^{*-}\|_{H^2(B_R\setminus B_{R_2})}+\|w^{*+}\|_{H^2(B_{R_1}\setminus B_R)}
\leq C\left(\|g^+\|_{L^2(\Omega^+)}+\|g^-\|_{L^2(\Omega^-)}\right).
\end{equation}
\end{lemma}

\begin{proof}[Proof of Lemma \ref{lem cru}]
We consider the Fourier series expansions:
\[
w^{*\pm}=a_0^\pm+ b_0^\pm \log r +\sum_{n\neq 0} (a^\pm_n r^n+b_n^\pm r^{-n}) e^{\,in\theta},
\]
and 
\[
\begin{aligned}
\phi^{*\pm}&=\sum_{n} \phi^\pm_{n} e^{\,in\theta}\\
h_{j}^*&=\sum_n h_{jn} e^{\,in\theta}
\end{aligned}
\]
The whole problem is now reduced to solving a linear system for each mode $n$. Consider first the case $n=0$. From (\ref{wmin sharp}) and (\ref{wplus sharp})  we get
\begin{equation}
\label{mode 0}
\left\{
\begin{aligned}
&a_0^- +b_0^- {\log R_2}=0\\
&a_0^- +b_0^- {\log R}=2h_{10}+2 h_{20}-\phi^-_{0}\\
&\frac{b_0^-}{R}=2\alpha h_{10}
\end{aligned}
\right.
\qquad 
\left\{
\begin{aligned}
&a_0^+ +b_0^+ {\log R_1}=0\\
&a_0^+ +b_0^+ {\log R}=2h_{10}-2h_{20}-\phi^+_0 \\
&\frac{b_0^+}{R}=-2\alpha h_{10}
\end{aligned}
\right.
\end{equation}
 Eliminating $a_0^\pm$, $b^\pm_0$ above we get
 \[
 \begin{aligned}
 h_{10}\left[\alpha R\log{R\over R_2}-1\right]-h_{20}&=-\frac{1}{2}\phi^-_0,\\
 h_{10}\left[\alpha R\log{R_1\over R }-1\right]+h_{20}&=-\frac{1}{2}\phi^+_0.
 \end{aligned}
 \]
This system is uniquely solvable for $h_{10}$ and $h_{20}$ when 
\begin{equation}
\label{con lamb 1}
\alpha R\log{R_1\over R_2}\neq {2}
\end{equation} 
which holds if
\[
 \alpha>\frac{{2}}{R\log{R_1\over R_2}}=\frac{{2}}{-R\log q}.
\] 
where $q=\frac{R_2}{R_1}<1$.\\
We have explicitly
\begin{equation}
\label{mode 0 1}
\begin{aligned}
h_{10}&=\frac{\phi_0^-+\phi_0^+}{2(\alpha R \log q+2)}=(\phi_0^-+\phi_0^+)\mathcal O(\alpha^{-1})\\
h_{20}&=\frac{1}{2}\phi_0^- +(\phi_0^-+\phi_0^+)\frac{\alpha R \log(R/R_2)-1}{2(\alpha R \log q+2)}=\phi^-_{0}\mathcal O(1)+\phi^+_{0}\mathcal O(1).
\end{aligned}
\end{equation}
Furthermore
\begin{equation}
\label{mode 0 2}
\begin{aligned}
b_0^+&=-2\alpha Rh_{10}= \phi^-_{0}\mathcal O(1)+\phi^+_{0}\mathcal O(1)\\
a_0^+&=-b_0^+ {\log R_1}=\phi^-_{0}\mathcal O(1)+\phi^+_{0}\mathcal O(1)\\
b_0^-&=2\alpha R h_{10}=\phi^-_{0}\mathcal O(1)+\phi^+_{0}\mathcal O(1)\\
a_0^-&=-b_0^- {\log R_2}=\phi^-_{0}\mathcal O(1)+\phi^+_{0}\mathcal O(1)
\end{aligned}
\end{equation}

Next we consider the Fourier modes $n\neq 0$.  
\begin{equation}
\label{mode n}
\left\{
\begin{aligned}
&a_n^- R_2^n +b_n^- R_2^{-n}=0\\
&a_n^- R^n +b_n^- R^{-n}=2h_{1n}+2 h_{2n}-\frac{1}{2}\phi_n^-\\
&n a^-_n R^{n-1}-n b^-_n R^{-n-1}=2\alpha h_{1n}
\end{aligned}
\right.
\qquad 
\left\{
\begin{aligned}
&a_n^+ R_1^n+b_n^+ R_1^{-n}=0\\
&a_n^+  R^n+b_n^+ R^{-n}=2h_{1n}-2 h_{2n}-\frac{1}{2}\phi^+_{n} \\
&na_n^+ R^{n-1}-nb_n^+ R^{-n-1}=-2\alpha h_{1n}
\end{aligned}
\right.
\end{equation}
With $q=\frac{R_2}{R_1}<1$
we get after eliminating $a_n^\pm, b_n^\pm$:
\begin{equation}
\label{mode n 0}
\begin{aligned}
h_{1n}\left(\frac{\alpha R}{n}\frac{1-q^n}{1+q^n}-1\right)-h_{2n}&=-\frac{1}{2}\phi^-_n\\
h_{1n}\left(\frac{\alpha R}{n}\frac{1-q^n}{1+q^n}-1\right)+h_{2n}&=-\frac{1}{2}\phi^+_{n}.
\end{aligned}
\end{equation}
This system can be uniquely solved for $f_{1n}, f_{2n}$ for any $\phi_n^\pm$ if
\begin{equation}\label{mode n 1}
\alpha\neq \frac{n}{R}\frac{1+q^n}{1-q^n}.
\end{equation}
This condition is the source of resonance. Let us denote
\begin{equation}
\alpha_n=\frac{n}{R} \frac{1+q^n}{1-q^n}.
\label{def alpha n}
\end{equation}
We know that 
$$\lim\limits_n(\alpha_{n+1}-\alpha_n)=\frac 1R\ \hbox{and}\ \lim\limits_n{\alpha_n}=+\infty.$$
Then given $\epsilon=\frac1{4R}$  there exists $N_{R, q}$ such that for any $n\ge N_{R, q}$ 
$$(\alpha_n+\epsilon,\alpha_{n+1}-\epsilon)\subset(\alpha_n,\alpha_{n+1}) \ \hbox{
and }\ \alpha_n\ge {1\over\epsilon}.$$
In particular, if $\alpha\in(\alpha_n+\epsilon,\alpha_{n+1}-\epsilon)$ then  
\begin{equation}
\label{3.5}
\min\limits_{n\ge N_{R,q}}|\alpha_n-\alpha|\ge \epsilon\ge\frac1{4R}.
\end{equation}

Using  (\ref{mode n 0}) and (\ref{3.5}) we get
\begin{equation}
\begin{aligned}
\label{mode n 2}
|h_{1n}| & \leq C\alpha^{-1}(|\phi^-_{n}|+|\phi^+_{n}|)\\
|h_{2n}| & \leq C(|\phi^-_{n}|+|\phi^+_{n}|).
\end{aligned}
\end{equation}
Furthermore we have
\begin{equation}
\label{mode n 3}
\left\{
\begin{aligned}
|b_n^-|&\leq C R_2^n(|\phi^-_{n}|+|\phi^+_{n}|),\\
|a_n^-|&\leq C R^{-n}(|\phi^-_{n}|+|\phi^+_{2n}|),
\end{aligned}
\right.
\qquad 
\left\{
\begin{aligned}
|b_n^+|&\leq C{ R_1^n(|\phi^-_{n}|+|\phi^+_{n}|),}\\
|a_n^+|&\leq C{R^{-n}(|\phi^-_{n}|+|\phi^+_{n}|).}
\end{aligned}
\right.
\end{equation}
From these estimates, summing up the Fourier modes we get (\ref{cru 1}). 
\end{proof}
Finishing now the proof of Proposition \ref{cru} is straightforward and is left to the reader. 
\end{proof}

The following corollary is immediate. 
\begin{corollary}
\label{cor cru}
Under the hypothesis of Proposition \ref{cru} and assuming additionally that $g^\pm\in H^1(\Omega^\pm)$ we have
\[
\|h\|_{H^2(\mathcal K_\gamma)}\leq C\left(\|g^+\|_{H^1(\Omega^+)}+\|g^-\|_{H^1(\Omega^-)}\right)
\]
\end{corollary}

\section{The inner linear problem}\label{lin theory}
\setcounter{equation}{0}

Consider the linearization of (\ref{liu 1})  around the inner solution 
\begin{align*}
\Delta+\lambda^2 e^{\,U(\lambda\mu_\lambda t)+2\log\mu_\lambda}.
\end{align*}
In the Fermi coordinates associated to the curve $\gamma$ we have:
\begin{align*}
\Delta\approx \Delta_\gamma+\partial^2_t-\varkappa \partial_t,
\end{align*}
where $\varkappa$ denotes the curvature of $\gamma$.  Motivated by this we see that
\[
\Delta+\lambda^2 e^{\,U(\lambda\mu_\lambda t)+2\log\mu_\lambda}\approx \partial^2_s+\partial_t^2+\lambda^2 e^{\,U(\lambda\mu_\lambda t)+2\log\mu_\lambda}.
\]
We make the following scaling of the variables $(s,t)$:
\begin{equation}
\label{fermi scaled}
 \xi=\alpha s, \quad \eta=\lambda\mu_\lambda t.
\end{equation}
Recall that $\lambda\mu_\lambda =\alpha|\partial_n \psi|$ (see  (\ref{fa 1})). 
In the scaled variables we have:
\begin{align*}
\partial^2_s+\partial_t^2+\lambda^2 e^{\,U(\lambda\mu_\lambda t)+2\log\mu_\lambda}\approx {(\lambda\mu_\lambda)^2}\left[\partial_\eta^2+\left(\frac{
\alpha}{\lambda\mu_\lambda}\right)^2\partial_\xi^2+ e^{\,U(\eta)}\right],
\end{align*}
where we have neglected small terms that come from differentiating  $\mu_\lambda=\mu_\lambda(s)$. We denote
\begin{align*}
p_\alpha(\xi):\,= \left(\frac{\alpha}{\lambda\mu_\lambda}\right)^2={\left|\partial_n\psi\left(\frac{\xi}{\alpha}\right)\right|^{-2}}.
\end{align*}
These considerations lead us to  the following linear problem:
\begin{align}
\label{lin model 1}
(\partial_\eta^2+p_\alpha(\xi)\partial_\xi^2)\phi+ e^{\,U(\eta)}\phi=g, \quad \mbox{in}\ \R\times [0,\alpha|\gamma|):\,={C}_{\alpha|\gamma|}.
\end{align}
Above   ${C}_{\alpha|\gamma|}$ is the straight cylinder of radius $\alpha|\gamma|$ or  in other words (\ref{lin model 1}) is to be considered with periodic boundary conditions.  The  operator on the left will be called the inner linear operator.

To solve (\ref{lin model 1}) we use separation of variables. To begin we  consider the eigenvalue problem
\begin{equation}
\label{eig alpha}
\begin{aligned}
&p_\alpha(\xi)\partial_\xi^2y_\alpha=-\omega^2_\alpha y_\alpha, \quad \mbox{in}\ (0,\alpha|\gamma|),\\
&y_\alpha(0)=y_\alpha(\alpha|\gamma|), \qquad y_\alpha'(0)=y_\alpha'(\alpha|\gamma|).
\end{aligned}
\end{equation}
Counting the eigenvalues with their  multiplicities we will denote the eigenvalues and the eigenfunctions respectively  by $\omega^2_{\alpha, k}$, $y_{\alpha,k}$, $k=0, 1,\dots$. 
We have:
\[
\omega_{\alpha, 0}=0, \quad \omega_{\alpha, k}>0, \quad k=1, \dots.
\]
Going back to the original variable $s$ by letting $y(s)=y_\alpha(\alpha s)$ we get the following related  eigenvalue problem:
\begin{equation}
\label{lin model 2 aaa}
\begin{aligned}
&\partial_s^2 y=-|\partial_n\psi (s)|^2\tilde \omega^2 y, \quad \mbox{in}\  (0,|\gamma|),\\
&y(0)=y(|\gamma|), \qquad y'(0)=y'(|\gamma|)
\end{aligned}
\end{equation}
(recall that $\lambda \mu_\lambda(s)=\alpha |\partial_n\psi(s)|$).
From this we get an obvious relation 
\begin{align*}
\omega_{\alpha,  k}^2= \left(\frac{\tilde\omega_k}{\alpha}\right)^2,
\end{align*}
where $\tilde \omega^2_k$ are the eigenvalues of (\ref{lin model 2 aaa}). In what follows we will normalize the eigenfunctions of (\ref{eig alpha}) setting  $\|y_{\alpha, k}\|_{L^2(\alpha|\gamma|)}=1$, which means $y_{\alpha, k}(\xi)\sim \frac{1}{\sqrt{\alpha}}y_k(\xi/\alpha)$ where $y_k$ are the normalized eigenfunctions of  (\ref{lin model 2 aaa}). 

We introduce the Liouville transformation
\[
\begin{aligned}
&\ell_0=\int_0^{|\gamma|} |\partial_n\psi(s)|\,ds, \qquad \tau=\frac{\ell_0}{\pi}\int_0^\tau |\partial_n\psi(s)|\,ds, \\
&\Psi(s)=|\partial_n\psi(s)|^{-1/2}, \qquad q(\tau)=\frac{\ell_0 \Psi''}{\pi^2\Psi |\partial_n\psi(s)|}.
\end{aligned}
\]
Under this transformation the eigenvalue problem (\ref{lin model 2 aaa}) becomes
\[
\begin{aligned}
&\partial_\tau^2 y+q(\tau) y=-\Lambda y,\\
&y(0)=y(\pi), \qquad y'(0)=y'(\pi),
\end{aligned}
\]
where $\Lambda=\frac{\ell_0^2}{\pi^2} \tilde \omega^2$.
Weyl's asymptotic formula implies that  as $k\to \infty$ we have:
\[
\Lambda_k=\left[2k +\mathcal O\left(\frac{1}{k^3}\right)\right]^2
\]
hence
\[
\tilde \omega^2_{k}=\left(\frac{2\pi k}{\ell_0}\right)^2+\mathcal O\left(\frac{1}{k^2}\right).
\]
We observe that along $\gamma$ we have, by the Cauchy-Riemann equations,  $|\partial_n\psi|^2=|\partial_\tau\psi|^2$ where $\partial_\tau$ is the derivative in the direction of the unit tangent. This implies
\begin{equation}
\label{lzero}
\ell_0=2\pi R=2\pi \sqrt{R_1 R_2}.
\end{equation}
Taking this into account we get
\begin{equation}
\label{lin model 2 bbb}
\omega_{\alpha,  k}^2=\left(\frac{k}{\alpha R}\right)^2+\mathcal O\left(\frac{1}{\alpha^2 k^2}\right).
\end{equation}
As we will see  this asymptotic behavior will be  relevant  when $k\sim \alpha$ with $\alpha\to \infty$.

If
\begin{align*}
g(\eta, \xi)=\sum_{k=0}^\infty g_k(\eta) y_{\alpha, k}(\xi),\, \qquad \phi(\eta, \xi)=\sum_{k=0}^\infty \phi_k(\eta) y_{\alpha, k}(\xi)
\end{align*}
the problem  (\ref{lin model 1})  decomposes into the following problems for each $k\geq 0$:
\begin{align}
\label{lin model 2}
(\partial_\eta^2 +e^{\, U(\eta)})\phi_k-\omega_{\alpha, k}^2 \phi_k=g_k(\eta), \quad \mbox{in}\ \R.
\end{align}
We denote:
\begin{align*}
\mathcal L_\omega = -(\partial_\eta^2 +e^{\, U(\eta)}-\omega^2) , \quad \omega\geq 0.
\end{align*}
We will now consider the spectrum of the operator $\mathcal L_\omega$.  The following is well known, see \cite{Tit}:
\begin{lemma}\label{lem lin model 1}
The operator $\mathcal L_0$ has a unique negative eigenvalue $\nu_0=-1$ with the corresponding eigenfunction $Z_0(\eta)=\sqrt{2}\sech \eta$. The continuous spectrum of $\mathcal L_0$  is the half line $[0,\infty)$.
\end{lemma}

Next we will study the fundamental set $\mathcal F_\omega$ of operator $\mathcal L_\omega$ with $\omega \in [0, \infty)$. When $\omega=0$ the fundamental set is
\[
\mathcal F_0=\mathrm{span}\, \{U', \eta U'+2\},
\]
(recall that $U'(\eta)=-2\tanh \eta$). Our goal is to understand $\mathcal F_\omega$, for the whole range $\omega\geq 0$.  Making change of variables $\phi(\eta)=y(\tanh\eta)$ turns $\mathcal L_\omega\phi=0$ into the Legendre equation
\[
(1-x^2)y'' -2xy'+\left(2-\frac{\omega^2}{1-x^2}\right) y=0, \qquad |x|<1.
\]
It is known that the endpoints $x=\pm 1$ are regular singular points of this equation and solutions can be found in terms of power series. The indicial roots are $\pm \frac{\omega}{2}$ and when $\omega>0$  near $x=1$ we can write 
\[
y^\pm (x)=(x-1)^{\pm \omega/2}\tilde y^\pm_\omega(x),
\]
where $\tilde y_\omega^\pm$ are smooth functions  of $x\in (1-r, 1]$ for any  $0<r<2$. Similar formula holds at $x=-1$.  These general observations will be used in the proof of the following:
\begin{lemma}\label{fund omega}
For each $\omega\geq 0$ the fundamental set $\mathcal F_\omega$ satisfies
\begin{itemize}
\item[(i)]
When  $\omega\geq \epsilon>0$ then $\mathcal F_\omega=\mathrm{span}\, \{k^+_\omega(\eta), k^-_\omega(\eta)\}$ where 
\[
|k_\omega^+(\eta)|\leq C e^{\,\omega\eta}, \qquad |k_\omega^-(\eta)|\leq C e^{\,-\omega\eta}, \qquad \eta\in \R,
\]
and $C>0$ does not depend on $\omega$. 
\item[(ii)]
When $0\leq \omega \leq \epsilon$ then $\mathcal F_\omega=\mathrm{span}\, \{k^+_\omega (\eta), l_\omega(\eta)\}$ and 
\[
k_\omega^+(\eta)= e^{\,\omega\eta} \varphi_\omega^+(\eta), 
\]
where $\varphi_\omega^+(\eta)\to U'(\eta)$ as $\omega\to 0$ uniformly in $\eta$, and 
\[
l_\omega(\eta)=\frac{e^{\,\omega\eta}-e^{\,-\omega\eta}}{2\omega}\varphi^-_\omega(\eta)+e^{-\omega\eta}\psi^-_\omega(\eta),
\]
where 
\[
\varphi^-_\omega(\eta)\to  U'(\eta), \qquad \psi^-_\omega(\eta)\to 2, \qquad \omega\to 0,
\]
uniformly in $\eta$.
\item[(iii)]
There exists a constant $C>0$ such that the Wronskian of the basis of $\mathcal F_\omega$ as in (i) or (ii) satisfies  $W_\omega\geq C(1+\omega)$ for all $\omega\geq 0$.  
\end{itemize}
\end{lemma}
\begin{proof}
From the equivalence between $\mathcal L_\omega$ and the Legendre equation it is not hard to show that 
\[
\mathcal F_\omega=\mathrm{span}\, \{k^+_\omega, k^-_\omega\}, 
\]
where
\[
|k^\pm_\omega(\eta)|\leq C_\omega e^{\,\pm\omega\eta}, \qquad \eta\in \R,
\]
so the issue is to show that we can chose $k_\omega^\pm$ in the case (i) and $k^+_\omega, l_\omega$ in the case (ii) in such a way that both $C_\omega$ and the Wronskian $W_\omega$  can be controlled   as claimed. 

Assume that $\omega>\epsilon$ as in (i). When additionally $\omega<M$, where $M>\epsilon$ is large then it is clear that $C_\omega$ can be chosen uniformly in  $\omega\in  [\epsilon, M]$ and we can arrange the functions $k^\pm_\omega$ in such a way that (iii) holds as well. Assume  that $\omega>M$, where $M$. Consider one of the functions in the fundamental set, say $k^+_\omega$. Supposing  that 
\[
k^+_\omega(\eta)=\tilde k_\omega (\omega\eta),  
\]
we have
\[
\tilde k''_{\omega}-\tilde k_\omega =-\frac{2}{\omega^2}\sech^2\left(\frac{x}{\omega}\right) \tilde k_\omega. 
\]
We can write
\[
\tilde k_\omega(x)= A_\omega  e^{\,x}+\frac{2e^{\,x}}{\omega^2}\int_x^\infty \sech^2\left(\frac{x'}{\omega}\right) \tilde k_\omega(x') e^{\,-x'}\,dx'+\frac{2e^{\,-x}}{\omega^2}\int_{-\infty}^x \sech^2\left(\frac{x'}{\omega}\right) \tilde k_\omega(x')e^{\,x'}\,dx'.
\]
Denote $K_\omega =\sup_x |\tilde k_\omega(x) e^{\,-x}|$.  It is not hard to show that for $\omega>M$ with $M$ sufficiently large
\[
K_\omega \leq {2 A_\omega}.
\]
From this we get 
\[
\tilde k_\omega(0)= A_\omega (1+\mathcal O(\omega^{-1})), \qquad \lim_{x\to\infty} \tilde k_\omega(x) e^{\,-x}= A_\omega(1+\mathcal O(\omega^{-1})). 
\]
Similar estimates can be proven for $\sup_x |\tilde k_\omega'(x) e^{\,-x}|$,  $\tilde k'_\omega(0)$ and $\lim_{x\to\infty} \tilde k'_\omega(x) e^{\,-x}$. At this point we are free to chose $A_\omega=1$. Going back to the original variable we get
\[
\sup_\eta |k^+_\omega(\eta) e^{\,-\omega\eta}|\leq 2, \qquad k^+_\omega(0)=1+\mathcal O(\omega^{-1}), \qquad \frac{d}{d\eta}k^+_\omega(0)=\omega +\mathcal O(1).
\]
We repeat the above argument  for $k^-_\omega(\eta)$ choosing its behavior at $-\infty$ in such a way that 
\[
\lim_{\eta \to\infty} k_\omega(\eta) e^{\,-\omega\eta}= -1+\mathcal O(\omega^{-1}).
\]
Then we also have
\[
\sup_\eta |k^-_\omega(\eta) e^{\,\omega\eta}|\leq 2, \qquad k^-_\omega(0)=-1+\mathcal O(\omega^{-1}), \qquad \frac{d}{d\eta}k^-_\omega(0)=-\omega +\mathcal O(1).
\]
Finally 
\[
W(k^+_\omega, k^-_\omega)=2\omega +\mathcal O(1).
\]
This proves (i) and (iii) in the case $\omega>\epsilon$.

Now we turn to the proof of (ii). Using the Legendre form of the operator $\mathcal L_\omega$ we find the expansion of $k^+_\omega$ in power series of $e^{\,-\omega\eta}$ as 
$\eta\to \infty$. We  get for $\eta>-M$, with $M$ large,
\[
k^+_\omega(\eta)=e^{\,\omega\eta}\varphi^+_\omega(\eta), 
\]
where $\varphi^+_\omega$ is a bounded function.   Without loss of generality we assume that $\lim_{\eta\to \infty}\varphi^+_\omega(\eta)=1$.
The Frobenius method applied with $\omega=0$  shows that 
\[
\varphi^+_\omega(\eta) \to U'(\eta)=\tanh \eta, \qquad \omega\to 0,
\]
uniformly in $\eta>-M$. It follows that
\[
k_\omega^+(\eta)\to U'(\eta), \qquad \omega\to 0,
\]
uniformly in  $[-M, M]$. The Frobenius method  at $x=-1$  (for the Legendre equation) gives finally 
\[
\varphi^+_\omega(\eta)\to U'(\eta), \qquad \omega \to 0,
\]
uniformly in $\eta \in \R$. To find $l_\omega$,  the second element in the fundamental set,   we observe that by the above argument we can show  
\[
k^-_\omega(\eta)=e^{\,-\omega\eta} \varphi^-_\omega(\eta), 
\]
where $\varphi^-_\omega(\eta)\to U'(\eta)$ as $\omega\to 0$, uniformly in $\eta\in \R$. We define 
\[
l_\omega(\eta)=\frac{k^+_\omega(\eta)-k^-_\omega(\eta)}{2\omega}.
\]
Clearly $\mathcal F_\omega=\mathrm{span}\,\{k^+_\omega, l_\omega\}$ and $l_\omega$ satisfies the assertions in (ii). Claim  (iii) in case $0\leq \omega\leq \epsilon$ also follows from this.   This ends the proof.
\end{proof}

Using the above Lemma we will study the following problem
\begin{equation}
\label{lem 3 eq 1}
\mathcal L_\omega \phi= g, 
\end{equation}
in the whole range of $\omega \geq 0$. We will work  in the weighted Sobolev spaces and assume $g\in L^2_\theta(\R)$ with some $\theta>0$ and look for $\phi\in H^2_\theta(\R)$. Before stating the next result we note that we will have to face at least two difficulties: first, when $\omega$ is small we are bound to  loose in the rate of exponential decay as in the case $\omega=0$; second, similar thing happens when $\theta= \omega$ due to the resonance of the non homogeneous problem (\ref{lem 3 eq 1}). 
\begin{lemma}\label{lem 3}
Let $\theta>0$ be fixed and let $\epsilon>0$ be a small number.
\begin{itemize}
\item[(i)] 
When $\omega>\theta+\epsilon$ then (\ref{lem 3 eq 1}) has a solution such that 
\begin{equation}
\label{lem 3 eq 2}
\|\phi\|_{H^2_\theta(\R)}\leq \frac{C}{1+\omega}\left(\frac{1}{|\omega-\theta|}+\frac{1}{|\omega+\theta|}\right)\|g\|_{L^2_\theta(\R)}.
\end{equation}
\item[(ii)] When $\epsilon<\omega<\theta-\epsilon$ and the following orthogonality condition is satisfied 
\begin{equation}
\label{lem 3 eq 3}
\int_{-\infty}^\infty k^+_\omega(\eta) g(\eta)\, d\eta=0=\int_{-\infty}^\infty k^-_\omega(\eta) g(\eta)\, d\eta,
\end{equation}
then there exists a solution to (\ref{lem 3 eq 1}) satisfying (\ref{lem 3 eq 2}). 
\item[(iii)] 
When $|\omega-\theta|\leq \epsilon$, the function $g$ is compactly supported $\supp g\subset [-A,A]$ and the orthogonality conditions (\ref{lem 3 eq 3}) hold  there exists a solution to (\ref{lem 3 eq 1}) such that 
\begin{equation}
\label{lem 3 eq 4}
\|\phi\|_{H^2_\theta(\R)}\leq C A(e^{\,2\epsilon A}+1)\|g\|_{L^2_\theta(\R)}. 
\end{equation}
\item[(iv)]
When $0\leq \omega\leq \epsilon$ and the orthogonality conditions (\ref{lem 3 eq 3}) are satisfied  there exists a solution of (\ref{lem 3 eq 1}) such that for any $\theta' <\theta-4\epsilon$
\begin{equation}
\label{lem 3 eq 5}
\|\phi\|_{H^2_{\theta'}(\R)}\leq \frac{C}{|\theta-\theta'|^2} \|g\|_{L^2_\theta(\R)}.
\end{equation}
\end{itemize}
\end{lemma}
\begin{proof}
Proofs of (i) and (ii) being similar we will concentrate on the latter.
We define the solution of (\ref{lem 3 eq 1}) by the formula
\begin{equation}
\label{lem 3 eq 6}
\phi(\eta)=-\frac{1}{W_\omega}\left(k_\omega^+(\eta)\int_{\eta}^\infty k_\omega^-(\eta') g(\eta')\,d\eta'+ k^-_\omega(\eta)\int_{-\infty}^\eta k^+_\omega(\eta') g(\eta')\,d\eta'\right).
\end{equation}
Let 
\[
Y(\eta)=k_\omega^+(\eta)\int_{\eta}^\infty k_\omega^-(\eta') g(\eta')\,d\eta'.
\]
We will estimate first $|e^{\,\theta|\eta|} Y(\eta)|$. When $\eta >0$ we have
\begin{equation}
\label{lem 3 eq 7}
\begin{aligned}
|e^{\,\theta|\eta|} Y(\eta)|&\leq C e^{\,(\theta+\omega)\eta}\int_\eta^\infty e^{\,-(\theta+\omega)\eta'} e^{\,\theta\eta'} |g(\eta')|\,d\eta'\\
&\leq C e^{\,(\theta+\omega)\eta}\left(\int_\eta^\infty e^{\,-2(\theta+\omega)\eta'}\,d\eta'\right)^{1/2}\|g\|_{L^2_\theta(\R)}\\
&\leq \frac{C}{\sqrt{\theta+\omega}} \|g\|_{L^2_\theta(\R)}.
\end{aligned}
\end{equation}
Using the orthogonality condition we get for $\eta<0$
\begin{equation}
\label{lem 3 eq 8}
|e^{\,\theta|\eta|} Y(\eta)|\leq \frac{C}{\sqrt{\theta-\omega}} \|g\|_{L^2_\theta(\R)}.
\end{equation}
Next we estimate
\[
\int_{-\infty}^\infty e^{\,2\theta|\eta|} Y^2(\eta)\,d\eta= \int_{-\infty}^0 e^{\,2\theta|\eta|} Y^2(\eta)\,d\eta+\int_{0}^\infty e^{\,2\theta|\eta|} Y^2(\eta)\,d\eta=I+II
\]
We will  consider  only the first term since  the bound on $II$ is analogous.   Using the orthogonality conditions and (\ref{lem 3 eq 7}),  (\ref{lem 3 eq 8}) we get
\[
\begin{aligned}
I&=\int_{-\infty}^0 e^{-2\theta \eta}\left (k^+_\omega (\eta)\int_{-\infty}^\eta k^-_\omega(\eta') g(\eta')\,d\eta'\right)^2\,d\eta\\
&\leq C\int_{-\infty}^0 e^{\,-2(\theta-\omega)\eta}\left(\int_{-\infty}^\eta e^{\,\omega\eta'} |g(\eta')|\,d\eta'\right)^2\,d\eta\\
&=-\frac{C}{2(\theta-\omega)} \lim_{y\to-\infty} \left.\left[e^{\,-2(\theta-\omega)}\left(\int_{-\infty}^\eta e^{\,\omega\eta'} |g(\eta')|\,d\eta'\right)^2\right]\right|_{y}^0\\
&\quad +\frac{C}{2(\theta-\omega)}\int_{-\infty}^0 e^{-2\theta \eta} e^{\,\omega\eta}|g(\eta)|\int_{-\infty}^0 e^{\,\omega\eta'} |g(\eta')|\,d\eta'd\eta\\
&\leq \frac{C}{(\theta-\omega)^2}\|g\|^2_{L^2_\theta(\R)}+\frac{C}{\theta-\omega}\|g\|_{L^2_\theta(\R)} \sqrt{I}.
\end{aligned}
\]
It follows 
\[
I\leq \frac{C}{(\theta-\omega)^2}\|g\|^2_{L^2_\theta(\R)}.
\]
Similar estimate can be shown for $II$ as well hence
\[
\|Y\|_{L^2_\theta(\R)}\leq C\left(\frac{1}{(\theta-\omega)}+\frac{1}{(\theta+\omega)}\right)\|g\|_{L^2_\theta(\R)}.
\]
Since the second term in (\ref{lem 3 eq 6}) can be bounded analogously using Lemma \ref{fund omega} we deduce (i) and (ii).

To show (iii), using the orthogonality conditions, we write 
\[
\begin{aligned}
\int_{-\infty}^\infty e^{\,2\theta|\eta|} Y^2(\eta)\,d\eta&=\int_{-\infty}^0 e^{\,-2\theta\eta}\left (k^+_\omega (\eta)\int_{-\infty}^\eta k^-_\omega(\eta') g(\eta')\,d\eta'\right)^2\,d\eta \\
&\quad +\int_{0}^\infty e^{\,2\theta\eta}\left (k^+_\omega (\eta)\int_{\eta}^\infty k^-_\omega(\eta') g(\eta')\,d\eta'\right)^2\,d\eta= III+IV.
\end{aligned}
\]
Since $g$ is compactly supported 
\[
\begin{aligned}
III&\leq C\int_{-A}^0 e^{\,-2\theta\eta}\left(e^{\,\omega\eta}\int_{-A}^\eta e^{\,-\omega\eta'}|g(\eta')|\,d\eta\right)^2\,d\eta\\
&= C\int_{-A}^0 e^{\,-2(\theta-\omega)\eta}\left(\int_{-A}^\eta e^{\,(\theta-\omega)\eta'} e^{\,-\theta\eta'}|g(\eta')|\,d\eta\right)^2\,d\eta\\
&\leq Ce^{\,4|\theta-\omega| A}A\left(\int_{-A}^0 e^{\,-\theta\eta'}|g(\eta')|\,d\eta'\right)^2\\
&\leq C e^{\,4|\theta-\omega|A}A^2\|g\|_{L^2_\theta(\R)}^2.
\end{aligned}
\]
Estimate (\ref{lem 3 eq 4}) follows from this and a similar bound for $IV$. 

To prove (iv) we write
\begin{equation}
\label{lem 3 eq 9}
\phi(\eta)=-\frac{1}{W_\omega}\left(k_\omega^+(\eta)\int_{\eta}^\infty l_\omega (\eta') g(\eta')\,d\eta'+ l_\omega(\eta)\int_{-\infty}^\eta k^+_\omega(\eta') g(\eta')\,d\eta'\right).
\end{equation}
From (ii) in Lemma \ref{fund omega} we get
\[
|l_\omega(\eta)|\leq C(|\eta|+1) e^{\,\omega|\eta|}.
\]
Based on this one can prove (iv) by an argument similar to that in the proof of (ii). We leave the details to the reader.
\end{proof}

We turn now to solving (\ref{lin model 1}), which we write shortly as
\[
L_{\xi, \eta}\phi=g, \quad \mbox{in}\ \R\times [0,\alpha|\gamma|):\,={C}_{\alpha|\gamma|},
\]
where
\[
L_{\xi, \eta}\phi=(\partial_\eta^2+p_\alpha(\xi)\partial_\xi^2)\phi+ e^{\,U(\eta)}\phi.
\]
Given $\psi\in L^2({C}_{\alpha|\gamma|})$ we write its Fourier expansion in terms of the normalized eigenfunctions $y_{\alpha, k}$ corresponding to  the eigenvalues $\omega_{\alpha, k}$:
\[
\psi=\sum_{k=0}^\infty \psi_k(\eta) y_{\alpha, k} (\xi)
\]
Given $\omega\geq 0$ we define the projections
\[
P_{>\omega}\psi=\sum_{\omega_{\alpha, k}>\omega} \psi_k(\eta) y_{\alpha, k} (\xi)
\]
and 
\[
P_{<\omega}=I-P_{>\omega}, \qquad P_{\omega_1<\omega_2}=P_{>\omega_1}-P_{>\omega_2}.
\]
We introduce the weighted Sobolev spaces $H^\ell_\theta(C_{\alpha|\gamma|})=e^{\,-\theta|\eta|}H^\ell(C_{\alpha|\gamma|})$.
\begin{lemma}
\label{lem 4}
Let $\epsilon>0$ be a small, fixed number. Given $g\in L^2_\theta(C_{\alpha|\gamma|})$ such that 
\[
g=\sum_k g_k(\eta) y_{\alpha, k}(\xi), 
\]
and the functions $g_k(\eta)$ satisfy (\ref{lem 3 eq 3}) whenever $\omega_{\alpha, k}<\theta+\epsilon$ 
there exists a solution $\phi$ to (\ref{lin model 1}) such that the following hold
\begin{itemize}
\item[(i)] When $\omega>\theta+\epsilon$
\[
\|P_{>\omega}\phi\|_{H^2_\theta(C_{\alpha|\gamma|})}\leq C\|P_{>\omega} g\|_{L^2_\theta(C_{\alpha|\gamma|})}.
\]
\item[(ii)]
When $\epsilon<\omega_1<\omega_2<\theta-\epsilon$
\[
\|P_{\omega_1<\omega_2}\phi \|_{H^2_\theta(C_{\alpha|\gamma|})}\leq C\|P_{\omega_1<\omega_2} g\|_{L^2_\theta(C_{\alpha|\gamma|})}.
\]
\item[(iii)]
When $|\omega-\theta|\leq \epsilon$ and the functions $g_k(\eta)$ with $|\theta - \omega_{\alpha, k}|\leq \epsilon$ are compactly supported, $\supp g_k\subset [-A, A]$ then
for any $\theta-\epsilon\leq \omega_1<\omega_2\leq \theta+\epsilon$
\[
\|P_{\omega_1<\omega_2}\phi\|_{H^2_\theta(C_{\alpha|\gamma|})}\leq C A(e^{\,2\epsilon A}+1)\|P_{\omega_1<\omega_2} g\|_{L^2_\theta(C_{\alpha|\gamma|})}.
\]
\item[(iv)]
When $0\leq \omega<\epsilon$ and $\theta'<\theta-4\epsilon$
\[
\|P_{0<\epsilon}\phi \|_{H^2_{\theta'}(C_{\alpha|\gamma|})}\leq \frac{C}{|\theta-\theta'|}\|P_{0<\epsilon} g\|_{L^2_\theta(C_{\alpha|\gamma|})}.
\]
\end{itemize}
\end{lemma}
\begin{proof}
To solve $L_{\xi, \eta}\phi=g$ we use separation of variables and Lemma \ref{lem 3}. This gives directly  assertions (i)--(iv) with $L^2_{\theta}(C_{\alpha|\gamma|})$ norms on the right hand side. To bootstrap to $H^1_{\theta}(C_{\alpha|\gamma|})$ we use the energy estimates. To bootstrap to $H^2_{\theta}(C_{\alpha|\gamma|})$ we use local elliptic estimates and a covering argument. This is rather  standard and we omit the details. 
\end{proof}

\section{Solution of the nonlinear problem-proof of  Theorem \ref{theorem liu}}
\setcounter{equation}{0}

\subsection{Set up of the fixed point argument}\label{sec nonlin setup}
\subsubsection{The inner region}

We recall that by $(s,t)$ we have denoted the Fermi coordinates of the curve $\gamma$  (see section \ref{subsec 1}). In what follows we will work with the scaled  version of them $(\xi, \eta)$ introduced in (\ref{fermi scaled}).
%,  namely near $\gamma$ we will use 
%\[
%\xi=\alpha s, \qquad \eta=\frac{t}{\lambda\mu_\lambda}, \qquad \lambda\mu_\lambda=\alpha|\partial_n\psi|.
%\] 
Note that $\eta=\eta(s,t)$ which introduces an extra complication. Recall that the map $X_\gamma$ 
\[
x\mapsto (s,t)\qquad x=\gamma(s)+tn(s),
\]
is a diffeomorphism in a neighborhood of $\gamma$, say when $|t|\leq \delta$ with some $\delta$ small. Likewise we define $X_{\gamma, \lambda}$ to be the map
$x\mapsto (\xi, \eta)$. In order to be precise about the meaning of the inner region, the inner problem etc. we introduce the function
\begin{equation}
\label{def del}
\delta_\lambda=\frac{M\log\beta}{\lambda\mu_\lambda},
\end{equation}
where $M>0$ is a constant that will be be adjusted later  as necessary. We speak of the inner region when $|t|\lesssim \delta_\lambda$ or $|\eta|\lesssim \log\beta$ (the scaling constant $\beta$ is defined in (\ref{beta})). Any function given in a neighborhood of $\gamma$ can be expressed in the global variable $x$ or in the stretched  Fermi variables $(\xi, \eta)$. To simplify the notation and to keep track of this ambiguity we will use the letters $u, v, w$ etc. for the functions of $x$ and the letters ${\tt u}, {\tt v}, {\tt w}$ etc. for the same functions expressed in $(\xi, \eta)$.  

\subsubsection{Modulation of the initial inner approximation}

Recall the first inner approximation $v_0$ defined in  (\ref{uzero 2}). Writing ${\tt v}_0=v_0\circ X^{-1}_{\gamma, \lambda}$ we have explicitly
\[
{\tt v}_0(\eta)=U(\eta)+2\log\mu_\lambda,
\]
with $\mu_\lambda=\mu_\lambda(s)=\mu_\lambda(\xi/\alpha)$ as in (\ref{mu}). The assertions  of Lemma \ref{lem 4} suggest that this approximation may not suffice since we need to satisfy the set of orthogonality conditions to solve the inner linear problem. To have more flexibility in the definition of ${\tt v}_0$ we will introduce the modulation of this function associated with the invariances of the Liouville equation. 

Given $\omega^*>0$ to be specified later we let $K_\alpha$ be the largest nonnegative integer  such that 
\[
\omega_{\alpha, k}<\omega^*, \qquad k\leq K_\alpha. 
\]
By (\ref{lin model 2 bbb}) we get
\begin{equation}
\label{asymp k}
K_\alpha=\alpha R{\omega}^*+\mathcal O\left(\frac{1}{\alpha}\right).
\end{equation}
Next we let $a(\xi)$, $b(\xi)$ to be $L^2(0,\alpha|\gamma|)$ functions of the form
\begin{equation}
\label{def ab}
a(\xi)=\sum_{k=0}^{K_\alpha} a_k y_{\alpha, k}(\xi), \qquad b(\xi)=\sum_{k=0}^{K_\alpha} b_k y_{\alpha, k}(\xi)
\end{equation}
and 
\begin{equation}
\label{def ff}
f(\xi, \eta)= [a(\xi)\eta+b(\xi)]\tanh\eta\sech\left(\frac{\eta}{Q}\right),
\end{equation}
where a constant $Q>0$ large is to be chosen, see section \ref{sec modulation} below.
The modulated inner initial approximation is defined by
\[
{\tt v}_0(\xi, \eta;f)= U(\eta+f)+2\log\left[\mu_\lambda(1+\partial_\eta f)\right].
\]
We will write ${\tt v}_0(f)$ whenever we need to indicate the dependence on the modulation parameter.  We will denote the space of the modulations by $\mathcal M$ and we will measure the size of the components $a, b$ of $f$ in the  ${K_\alpha}$ dimensional subspace of $L^2(0, \alpha|\gamma|)$: 
\[
\|f\|_{L^2(\mathcal M)}=\left(\sum_{k=0}^{K_\alpha} |a_k|^2+|b_k|^2\right)^{1/2}=\left(\|a\|^2_{L^2(0, \alpha|\gamma|)}+\|b\|^2_{L^2(0, \alpha|\gamma|)}\right)^{1/2}.
\]
Similarly we define
\[
\|f\|_{H^2(\mathcal M)}=\left(\|a\|^2_{H^2(0, \alpha|\gamma|)}+\|b\|^2_{H^2(0, \alpha|\gamma|)}\right)^{1/2}.
\]
We will suppose {\it a priori}
\begin{equation}
\label{fix 1}
\|f\|_{H^2(\mathcal M)}\leq \frac{\log^q\beta}{\sqrt{\beta}}
\end{equation}
where $q>2$ is to be specified later. Note that under this hypothesis we have as well
\begin{equation}
\label{f linfty}
\|f\|_{L^\infty(\mathcal M)}=\|a\|_{L^\infty(0, \alpha|\gamma|)}+\|b\|_{L^\infty(0, \alpha|\gamma|)}\lesssim \frac{\log^q\beta}{\sqrt{\beta}}.
\end{equation}

\subsubsection{Additive perturbation of  the inner approximation}\label{sec add}
With the notation introduced in section \ref{general improvement} we let $\bar {\tt v}_1\in \mathcal K_\gamma$. 
We have $\bar {\tt v}_1= \bar{\tt  v}_1(s,\lambda\mu_\lambda t)=\bar{\tt v}_1(s, \eta)$ and  $\bar v_1(x)=\bar {\tt v}_1\circ X_{\gamma, \lambda}(x)$.   We recall that 
\[
\bar {\tt v}_1(s, \eta)=h_1(s)\varphi_1(\eta)+h_2(s)\varphi_2(\eta)
\]
and that 
\[
\|\bar {\tt v}_1\|^2_{L^2(\mathcal K_\gamma)}=\|h_1\|^2_{L^2(\gamma)}+\|h_2\|^2_{L^2(\gamma)}, \qquad \|\bar {\tt v}_1\|^2_{H^\ell (\mathcal K_\gamma)}=\|h_1\|^2_{H^\ell(\gamma)}+\|h_2\|^2_{H^\ell(\gamma)}.
\]

Let $\theta>0$, to be specified later,  represent the rate of exponential decay. We consider 
\begin{equation}
\label{add 2}
{\tt v}_2\in H^2_\theta(C_{\alpha|\gamma|}). 
\end{equation}
Note that  $\bar {\tt v}_1$ is an affine function of $\eta$  while ${\tt v}_2$ is exponentially decaying.  We will assume that with some $\sigma\in (0,1/4)$ we have
\begin{equation}
\label{fix 2}
\begin{aligned}
\|\bar {\tt v}_1\|_{H^2(\mathcal K_\gamma)}&\leq \beta^{-1+\sigma}\\
\|{\tt v}_2\|_{H^2_\theta(C_{\alpha|\gamma|})} &\leq \beta^{-1/2+\sigma}. 
\end{aligned}
\end{equation}

\subsubsection{Perturbation of the outer approximation}

Recall the initial outer approximation defined in (\ref{wzero}). Recall also that parameters $\beta$ and $\mu_\lambda$ were defined in such a way that matching conditions (\ref{match 1a}) were approximately satisfied. With the  given initial outer approximation  $w^\pm_0$ on $\Omega^\pm$ it is natural to suppose that the true solution in the outer region should be of the form
\[
w^\pm =w_0^\pm+w_1^\pm, 
\]
where {\it a priori} we assume, with some $\sigma\in (0,1/4)$
\begin{equation}
\label{fix 3}
\|w_1^\pm\|_{H^2(\Omega^\pm)} \leq \beta^{-1+\sigma}.
\end{equation}
The matching conditions between the inner and the outer solution is essentially of the  form of  the boundary conditions in (\ref{3.1}) and it will be satisfied automatically because we will use Proposition \ref{cru} to determine $w^\pm_1$ in the course of the fixed point scheme.  A precise form of the matching conditions  suitable for the fixed point argument  is not easy to give at this moment and will be specified later (see section \ref{sec patch}).    

\subsubsection{Global definition of the solution}

To define the solution globally we will patch together  the inner and the outer solution using some smooth cutoff functions. This turns out to be the  delicate part of the construction. 
All cutoff functions introduced  below will be of the form $\chi(t/\delta_\lambda)$ where $\delta_\lambda$ is defined in (\ref{def del}) and  $\chi(t/\delta_\lambda)$ is  supported in a neighborhood of with $|t|< 2m\delta_\lambda$, $m>0$ around $\gamma$. In the inner variables $(\xi, \eta)$ the function $\chi(\eta/\lambda \mu_\lambda)$ is supported in the set $|\eta|<2m\delta_\lambda \lambda \mu_\lambda=2mM\log\beta$. With some abuse of notation we will use the same letters to denote the cutoff functions expressed in $x$ or in $t$ or in $\eta$ as the form of the dependence will be clear from the context.

First we define $\chi_0$, $\chi_0^\pm$ so that 
\[
\supp\chi_0\subset \{|t|<2\delta_\lambda\}, \quad \supp \chi^+_0\subset  \{t>\delta_\lambda\}, \quad \supp \chi^-\subset \{t<-\delta_\lambda\},
\]
and additionally $\chi_0+\chi_0^++\chi_0^-\equiv 1$, $\chi_0=1$ when $|t|<\delta_\lambda$. Let ${\tt v}_0(f)$ be the inner initial approximation and let 
$v_0(x;f)={\tt v}_0(f)\circ X_{\gamma, \lambda}(x)$. We define the global initial approximation by
\begin{equation}
u_0=\chi_0 v_0+\chi^+_0 w^+_0+\chi^-_0 w^-_0.
\label{fix 4}
\end{equation}
Note that $u_0$ depends on the modulation parameter $f$ and for this reason we will write sometimes $u_0(x; f)$. 

Next we introduce cutoff functions $\chi_1$, $\chi_1^\pm$. They have similar character as $\chi_0$, $\chi_0^\pm$ but their supports are slightly different. Let $m_1\in \R$, $m_1>2$, be  fixed. We now define $\chi_1$ to be a smooth cutoff function supported in $|t|<2m_1\delta_\lambda$ and such that $\chi_1=1$ in $|t|<m_1\delta_\lambda$. Let $\chi^\pm_1$ be smooth cutoff functions such that: 
\[
\chi_1+\chi^+_1+\chi_1^-\equiv 1,\quad \mbox{in}\ \Omega.
\]
Let $\bar {\tt v}_1 \in \mathcal K_\gamma$ as in (\ref{fix 2}) and let $\bar v_1(x)=\bar {\tt v}_1\circ X_{\gamma, \lambda}(x)$.  Let $w_1^\pm$ be as in (\ref{fix 3}). We define 
\begin{equation}
\label{fix 5}
u_1=\chi_1\bar  v_1+\chi^+_1 w_1^++\chi^-_1 w_1^-.
\end{equation}
Finally  let  $\chi_{2}$ be a cutoff function supported in $|t|<2 m_2\delta_\lambda $ and $\chi_2=1$ in $|t|<m_2\delta_\lambda$, $m_2>m_1$.  Let $v_2(x)={\tt v}_2\circ X_{\gamma, \lambda}(x)$, where ${\tt v}_2$ satisfies  (\ref{add 2}). We define 
\begin{equation}
\label{fix 6}
u_2=\chi_{2}v_2.
\end{equation} 
We will look for a solution of the Liouville equation (\ref{liu 1}) in the form 
\begin{equation}
\label{fix 7}
u=u_0+u_1+u_2,
\end{equation}
%For the purpose of the fix point argument this function can be interpreted as  a map
%\[
%(f, \bar v_1, v_2, w_1^+, w^-_1)\longmapsto u
%\]  
%which gives the nonlinear map 
%\[
%(f, \bar v_1,\tilde v_1,  v_2, w_1^+, w^-_1)\longmapsto N_\lambda(u)
%\]

\subsection{The error of the initial approximation}
\subsubsection{Change of variables in $\Delta$}
In this section we will study the  initial approximation  $u_0$.  To begin it will be useful to express  $\Delta$ in terms of the Fermi variables $(s,t)$ and $(\xi, \eta)$. Sometimes we will  denote these different expressions by $\Delta_x$, $\Delta_{s,t}$, $\Delta_{s,\eta}$ and $\Delta_{\xi, \eta}$. Let $\varkappa$ denote the curvature of $\gamma$ and let $A=(1-t \varkappa(s))^2$. Then we have:
\begin{equation}
\label{fix 8}
\Delta_{s,t} =\partial_t^2+\frac{1}{A}\partial_s^2-\frac{\partial_s A}{2A^2}  \partial_s+\frac{\partial_t A}{2A}\partial_t.
\end{equation}
From this we find
\begin{equation}
\label{change variables}
\begin{aligned}
\Delta_{s,\eta} &=(\lambda\mu_\lambda)^2\left(\partial^2_{\eta}-\frac{\varkappa}{\lambda\mu_\lambda} \partial_\eta\right)\\
&\quad + (\lambda\mu_\lambda)^2\left(a_{20}\partial^2_{\eta}+a_{11} \partial_{\eta s}+a_{02}\partial^2_{s}+b_1\partial_\eta+b_2\partial_s\right),
\end{aligned}
\end{equation}
where, as long as 
\[
|\eta|\lesssim \delta_\lambda\lambda\mu_\lambda\sim \mathcal O(\log\beta).
\] 
we have
\begin{equation}
\label{ch var 2}
\begin{aligned}
a_{20}&=\frac{1}{a}\left(\frac{\partial_s\mu_\lambda}{\mu_\lambda}\right)\left(\frac{\eta}{\lambda\mu_\lambda}\right)^2=\mathcal O(\beta^{-2}\log^2\beta), \qquad a_{11}=\frac{2}{a}\frac{\partial_s\mu_\lambda}{\mu_\lambda}\frac{\eta}{(\lambda\mu_\lambda)^2}=\mathcal O(\beta^{-2}\log\beta),\\
a_{02}&=\frac{1}{a}\frac{1}{(\lambda\mu_\lambda)^2}=\mathcal O(\beta^{-2})\\
b_1&=\left(1-\frac{1}{a}\right) \frac{\varkappa}{\lambda\mu_\lambda}+\frac{1}{(\lambda\mu_\lambda)^2}\left[\frac{\varkappa'}{a}\frac{\partial_s\mu_\lambda}{\mu_\lambda}\frac{\eta^2}{\lambda\mu_\lambda}+\frac{\eta}{a}\partial_s\left(\frac{\partial_s\mu_\lambda}{\mu_\lambda}\right)\right]=\mathcal O(\beta^{-2}\log\beta),\\
b_2&=\frac{1}{(\lambda\mu_\lambda)^2}\frac{\varkappa'}{a}\frac{\eta}{\lambda\mu_\lambda}=\mathcal O(\beta^{-3}\log\beta).
\end{aligned}
\end{equation}
Recalling that $\xi=\alpha s$ we get
\begin{equation}
\label{fix 9}
\begin{aligned}
\Delta_{\xi, \eta}&=(\lambda\mu_\lambda)^2\left(\partial^2_{\eta}+p_\alpha(\xi)\partial^2_\xi\right)-\lambda\mu_\lambda \varkappa \partial_\eta\\
&\quad + (\lambda\mu_\lambda)^2 \left(\tilde a_{20}\partial^2_{\eta}+\tilde a_{11} \partial_{\eta \xi}+\tilde a_{02}\partial^2_{\xi}+\tilde b_1\partial_\eta+\tilde b_2\partial_\xi\right)  
\end{aligned}
\end{equation}
where, changing $s$ to $\xi/\alpha$ in the above formulas we get for the corresponding coefficients
\[
\begin{aligned}
\tilde a_{20}&=\mathcal O(\beta^{-2}\log^2\beta), \qquad  \tilde a_{11}=\mathcal O(\beta^{-1}\log\beta),\qquad
\tilde a_{02}=\mathcal O(\beta^{-1}\log\beta), \\
 \tilde b_1&=\mathcal O(\beta^{-2}\log\beta),\qquad
\tilde b_2=\mathcal O(\beta^{-2}\log\beta).
\end{aligned}
\]
To simplify the notation the operators in the second line in  (\ref{change variables}) and (\ref{fix 9}) will be denoted respectively by ${A}_\beta$ and $\tilde {A}_\beta$.  

\subsubsection{Solvability of the modulation equation}\label{sec modulation}

In what follows  we will denote
\[
N_\lambda (u)=\Delta u+\lambda^2 e^{\,u}, \quad {\tt N}_\lambda({\tt u})=N_\lambda (u)\circ X^{-1}_{\gamma, \lambda}.
\]
We use (\ref{fix 9})  to calculate
\[
{\tt N}_\lambda({\tt v}_0(\cdot; f)):=\Delta_{\xi, \eta} {\tt v}_0(\xi, \eta;f)+\lambda^2 e^{\, {\tt v}_0(\xi, \eta; f)}.
\]
Dividing by $(\lambda\mu_\lambda)^2$ we get
\begin{equation}
\label{fix 10}
\begin{aligned}
\frac{1}{(\lambda\mu_\lambda)^2}{\tt N}_\lambda({\tt v}_0(\cdot; f))&=\left[U'\left(\partial_\eta^2 f- \frac{\varkappa}{\lambda\mu_\lambda}\right)+{2\partial^3_\eta f}\right]
\\
&\quad +2\left[{-\frac{U'}{2} \frac{\varkappa\partial_\eta f}{\lambda\mu_\lambda}-\frac{\varkappa}{\lambda\mu_\lambda}\frac{\partial_\eta^2 f}{1+\partial_\eta f}-\frac{\partial_\eta^3 f\partial_\eta f}{1+\partial_\eta f}-\frac{(\partial_\eta^2 f)^2}{(1+\partial_\eta f)^2} }\right]\\
&\quad +p_\alpha\left[U'' (\partial_\xi f)^2+U' \partial_\xi^2 f+2\partial_\xi\left(\frac{\partial_\xi\mu_\lambda}{\mu_\lambda}\right)+\frac{2\partial_{\xi\xi\eta} f}{1+\partial_\eta f}-\frac{2(\partial_{\xi\eta} f)^2}{(1+\partial_\eta f)^2}\right] +\tilde{A}_\beta {\tt v_0}(\cdot; f),
%\\ &=N_1+N_2+N_3+N_4.
\end{aligned}
\end{equation}
above $U'=U'(\eta+f)$, $ U''=U''(\eta+f)$. 
We will study more carefully the leading (linear) term carrying the derivatives of $f$ in (\ref{fix 10}) taking now $U'=U'(\eta)$, 
\begin{equation}
\label{fix 11}
\begin{aligned}
&U'(\eta)\partial^2_\eta f+2\partial^3_\eta f +p_\alpha( U'(\eta)\partial_\xi^2 f+2\partial_{\xi \xi\eta} f)\\
&\qquad =:
a(\xi) {\tt Z}_1(\eta)+ p_\alpha(\xi)\partial_\xi^2 a(\xi) {\tt W}_1(\eta)+ b(\xi) {\tt Z}_2(\eta)+p_\alpha(\xi)\partial_\xi^2 b(\xi) {\tt W}_2(\eta),
\end{aligned}
\end{equation}
where, denoting $Y(\eta)=\tanh\eta\sech\left(\frac{\eta}{Q}\right)$ (see (\ref{def ff})) we have
\[
\begin{aligned}
{\tt Z}_1&=\eta U'(\eta) Y''+2U'(\eta)Y'+2\eta Y'''+6 Y'',\\
{\tt W}_1&= \eta U'(\eta) Y+2Y+2\eta Y', \\
{\tt Z}_2&= U'(\eta) Y'' +2Y''',\\
{\tt W}_2&= U'(\eta) Y+2Y'.
\end{aligned}
\]
The following lemma can be checked either by an explicit  calculation or by using a software capable of symbolic and numerical  calculations such as Maple or Mathematica.
\begin{lemma}
\label{lem aux 1}
\begin{itemize}
\item[(i)]
When $Y(\eta)=\tanh \eta$ (i.e. $Q=+\infty$) then 
\[
\begin{aligned}
&\int_\R {\tt Z}_1(\eta) U'(\eta)\,d\eta=-\int_\R {\tt Z}_2(\eta) (\eta U'(\eta)+2)\,d\eta=8,\\
&\int_\R {\tt Z}_2(\eta) U'(\eta)\,d\eta=-\int_\R {\tt Z}_1(\eta) (\eta U'(\eta)+2)\,d\eta=0.
\end{aligned}
\]
\item[(ii)]
When $Y(\eta)=\tanh \eta\sech\left(\frac{\eta}{Q}\right)$, $0<Q<\infty$  then
\[
\begin{aligned}
&\int_\R {\tt W}_1(\eta) U'(\eta)\,d\eta=-\int_\R {\tt W}_2(\eta) (\eta U'(\eta)+2)\,d\eta,\\
&\int_\R {\tt W}_2(\eta) U'(\eta)\,d\eta=-\int_\R {\tt W}_1(\eta) (\eta U'(\eta)+2)\,d\eta=0.
\end{aligned}
\]
Moreover when $Y(\eta)=\tanh\eta$ then 
\begin{multline*}
\lim_{R\to\infty}\left\{\int_{-R}^R \left[{\tt W}_1(\eta) U'(\eta)-4(\eta\tanh \eta-1)\right]\,d\eta\right\}\\
=\lim_{R\to\infty}\left\{\int_{-R}^R \left[{\tt W}_2(\eta) (\eta U'(\eta)+2)-4(\eta\tanh \eta-1)\right]\,d\eta\right\}=\ell
\end{multline*}
where $|\ell|<\infty$. 
\end{itemize}
\end{lemma}

To solve the modulation equations we will further need the following: 
\begin{lemma}
\label{lem fix 1}
Let $k^\pm_\omega$, $l_\omega$ be the elements of the fundamental set of $\mathcal L_\omega$, as in Lemma  \ref{fund omega}. Let $\epsilon>0$ be small. There  exist $0<Q<\infty$ and  $\omega^*>\epsilon$ such that
\begin{itemize}
\item[(i)]
For any $0\leq \omega \leq \epsilon$ the matrices  
\[
\begin{aligned}
{Z}_\omega := \left[\begin{array}{cc}
\int_\R {\tt Z}_1 k_\omega^+ & \int_\R {\tt Z}_2 k_\omega^+\medskip \\
\int_\R {\tt Z}_1 l_\omega & \int_\R {\tt Z}_2 l_\omega
\end{array}
\right], 
&
\qquad {W}_\omega := \left[\begin{array}{cc}
\int_\R {\tt W}_1 k_\omega^+ & \int_\R {\tt W}_2 k_\omega^+\medskip \\
\int_\R {\tt W}_1 l_\omega & \int_\R {\tt W}_2 l_\omega
\end{array}
\right],
\end{aligned}
\]
are  invertible and the norm  of the inverse is independent on $\omega$.
\item[(ii)]
The same holds when $\epsilon< \omega\leq \omega^*$ for the matrices
\[
{Z}_\omega := \left[\begin{array}{cc}
\int_\R {\tt Z}_1 k_\omega^+ & \int_\R {\tt Z}_2 k^+_\omega\medskip \\
\int_\R {\tt Z}_1 k_\omega^- & \int_\R {\tt Z}_2 k^-_\omega
\end{array}
\right], 
\qquad {W}_\omega := \left[\begin{array}{cc}
\int_\R {\tt W}_1 k_\omega^+ & \int_\R {\tt W}_2 k^+_\omega\medskip \\
\int_\R {\tt W}_1 k_\omega^- & \int_\R {\tt W}_2 k^-_\omega
\end{array}
\right].
\]
\end{itemize}
\end{lemma}
\begin{proof}
We prove (i).  Regarding invertibility of the matrix $Z_\omega$, using Lemma \ref{fund omega} (i),  we see that it suffices to show that 
\[
{Z}_0 := \left[\begin{array}{cc}
\int_\R {\tt Z}_1 U' & \int_\R {\tt Z}_2 U'\medskip \\
\int_\R {\tt Z}_1 (\eta U'+2) & \int_\R {\tt Z}_2 (\eta U'+2)
\end{array}
\right]
\]
is invertible. This follows easily from Lemma \ref{lem aux 1}.  We argue similarly for $W_\omega$. To show (ii) we note that subtracting the first column from the second column in ${Z}_\omega$ and dividing by $-\omega$ we get the matrix of the same form as in (i) (c.f. proof of Lemma \ref{fund omega} (ii)), denote it by $\tilde {Z}_\omega$. The  new matrix is invertible if and only if the original one is and the norms of the inverses differ by a factor $\omega^{-1}$, which is bounded since $\omega\geq \epsilon$. As a function of $\omega$ the matrix ${\tilde{Z}}_\omega$ is continuous and by (i) $\tilde {Z}_\omega$ is boundedly invertible for $0\leq \omega\leq \epsilon$. As a consequence there exists $\omega^*>\epsilon$ such that ${Z}_\omega$ is invertible  for $\epsilon\leq \omega\leq \omega^*$. Same argument applies for $W_\omega$. The proof is complete.
\end{proof}

Let  $K_\alpha$ be  such that $\omega_{\alpha, k}<\omega^*$ when $k\leq K_\alpha$. 
For $k\leq K_\alpha$ project the right hand side of (\ref{fix 11}) on $k^+_\omega y_{\alpha, k}$, $l_\omega y_{\alpha, k}$ when $0\leq\omega_{\alpha, k} \leq\epsilon$ and on
$k^+_\omega y_{\alpha, k}$, $k^-_\omega y_{\alpha, k}$ when $\epsilon<\omega_{\alpha, k} \leq\omega^*$. Then for each $k\leq K_\alpha$ the projection of (\ref{fix 11}) becomes
\begin{equation}
\label{fix 12}
\left({Z}_{\omega_{\alpha, k}} -\omega^2_{\alpha, k} {W}_{\omega_{\alpha, k}}\right)(a_k, b_k)^T
\end{equation}
Taking $\omega^*$ smaller if necessary, by Lemma \ref{lem fix 1} we conclude that the matrix ${Z}_{\omega_{\alpha, k}} -\omega^2_{\alpha, k} {W}_{\omega_{\alpha, k}}$ is boundedly invertible.  From this we readily  deduce:
\begin{corollary}
\label{cor fix 1}
Let $\theta>\omega^*$ and ${\tt h} \in L^2_\theta(C_{\alpha|\gamma|})$ be given. For each $k\leq K_\alpha$ we define
\begin{equation}
\label{c lem fix 2}
{\tt h}_{1k}=\int_{C_{\alpha|\gamma|}} {\tt h}(\xi, \eta) k^+_{\omega_{\alpha, k}}(\eta) y_{\alpha, k}(\xi)\,d\xi d\eta, \qquad {\tt h}_{2k}=\begin{cases} 
\int_{C_{\alpha|\gamma|}} {\tt h}(\xi, \eta) l_{\omega_{\alpha, k}}(\eta) y_{\alpha, k}(\xi)\,d\xi d\eta, \ 0<\omega_{\alpha, k}\leq \epsilon,\medskip\\
\int_{C_{\alpha|\gamma|}} {\tt h}(\xi, \eta) k^-_{\omega_{\alpha, k}}(\eta) y_{\alpha, k}(\xi)\,d\xi d\eta, \ \epsilon<\omega_{\alpha, k}<\omega^*.
\end{cases}
\end{equation}
There exist $a_k, b_k$ such that 
\[
\left({Z}_{\omega_{\alpha, k}} -\omega^2_{\alpha, k} {W}_{\omega_{\alpha, k}}\right)(a_k, b_k)^T=({\tt h}_{1k}, {\tt h}_{2k})
\]
Additionally, considering $(a_k, b_k)$ as the coefficients of the Fourier expansion of   the modulation function $f$ defined  in  (\ref{def ab})--(\ref{def ff}) we have
\[
\|f\|_{H^2(\mathcal M)}\leq C\|{\tt h}\|_{L^2_\theta(C_{\alpha|\gamma|})}.
\]
\end{corollary}

To justify the assumption  we made in (\ref{fix 1}) we go back to (\ref{fix 10}) and note that the leading term on the right hand side not involving $f$ is $- \frac{U'\varkappa}{\lambda\mu_\lambda}$. Clearly this function does not decay in $\eta$ but in the process of solving our problem it will appear multiplied by a cutoff function whose support is in the set $|\eta|\lesssim \log\beta$. Let $\chi(\eta)$ be such a function. Then 
\begin{equation}
\label{size f}
\left|\int_{C_{\alpha|\gamma|}} \chi \frac{U'\varkappa}{\lambda\mu_\lambda} U' y_{\alpha, k}\right|+\left|\int_{C_{\alpha|\gamma|}}\chi \frac{U'\varkappa}{\lambda\mu_\lambda} (\eta U'+2) y_{\alpha, k}\right|\lesssim \frac{\log^2\beta}{\sqrt{\beta}},
\end{equation}
consistently with (\ref{fix 1}). Similar estimates hold when $U'$ is replaced by $k_\omega^+$ and $\eta U'+2$ by $l_\omega$ or $k^-_\omega$.  

For the rest of this paper the constant $Q$ in the definition of the function $Y=\tanh\eta\sech\left(\frac{\eta}{Q}\right)$ will be fixed in such a way that the assertion of Corollary \ref{cor fix 1} holds. In particular we observe that if $\theta<\frac{1}{Q}$ then $f\in H^2_\theta(C_{\alpha|\gamma|})$ and 
\[
\|f\|_{H^2_\theta(C_{\alpha|\gamma|})}\sim \|f\|_{H^2(\mathcal M)}.
\]

\subsubsection{Estimates for the  error of  the initial approximation}
With $u_0$ defined in (\ref{fix 4}) we have
\begin{equation}
\begin{aligned}
\label{fix 14}
N_\lambda(u_0)&=\chi_0(\Delta v_0+\lambda^2 e^{\,v_0})+\left[\Delta, \chi_0^+\right] (w^+_0- v_0)+\left[\Delta, \chi_0^-\right] (w^-_0 - v_0)
+\lambda^2 \left(e^{\,u_0}-\chi_0 e^{\, v_0}\right),
\end{aligned}
\end{equation}
where here and elsewhere $\left[\Delta, \chi\right]=\Delta\chi-\chi\Delta$. 
To help dealing with the exponential part   of the error we prove:
\begin{lemma}\label{lem fix 2}
There exists a constant $c>0$ such that for any $m>0$ 
\[
w_0^\pm(x) \leq \beta(1- cm\delta_\lambda)\left[1+\mathcal O\left(\frac{\log\beta}{\beta}\right)\right],\quad  \hbox{in}\ \Omega^\pm\setminus \{|t|<m\delta_\lambda\},
\]
where $t$ is the Fermi coordinate of $x$.
\end{lemma}

\begin{proof}
For brevity we suppose that $m=1$ since the proof does not depend on the particular value of $m$. We claim that there exists a constant $c= c(\gamma, \Omega)>0$ such that
\begin{equation}
\label{hopf}
0\le H^\pm_\gamma(x)\leq 1- c\delta_\lambda, \quad x\in \Omega^\pm\cap\{ \mathrm{dist}\,(x,\gamma)\ge\delta_\lambda\}.
\end{equation}
In the  set  $\Omega^\pm \cap\{  \mathrm{dist}\,(\gamma, x)\le \delta_\lambda\} ,$  we can write $x=\pi(x)+d(x)\nu^\pm \in\Omega^\pm,$ where $\pi(x)\in\partial\Omega^\pm=\gamma,$   $d(x)=|\pi(x)-x|$ and $\nu^\pm $ is the inner normal vector    at $\partial\Omega^\pm.$ 
By Taylor's expansion  we get
$$\begin{aligned}
H^\pm_\gamma(x)&=H^\pm_\gamma(\pi(x))+\partial_{\nu^\pm} H^\pm_\gamma(\pi(x))d(x)+\mathcal O(d(x)^2) 
 \le 1 -c d(x),\end{aligned} $$
 because $H^\pm_\gamma\equiv1$ on $\gamma$ and by Hopf's lemma $0>\max\limits_\gamma \partial_{\nu^\pm} H^\pm_\gamma=:-c (\gamma, \Omega)$.
This shows
  $$H^\pm_\gamma(x)\le   1 -c\delta_\lambda\quad  \hbox{in}\quad  \Omega^\pm \cap\{  \mathrm{dist}\,(\gamma, x)= \delta_\lambda\}.$$
  The function  $H^\pm_\gamma$ is harmonic in $\Omega^\pm \setminus \{  \mathrm{dist}\,(\gamma, x)\le \delta_\lambda\},$ so it achieves its maximum and minimum values on the boundary of this set and this proves  \eqref{hopf}.

By \eqref{wzero} and \eqref{hopf} we deduce 
\[
\begin{aligned}
w_0^\pm &\leq \beta  H_\gamma^\pm  \left[1+\mathcal O\left(\frac{\log\beta}{\beta}\right)\right] +\mathcal O(1)\\
&\leq \beta(1- c\delta_\lambda)\left[1+\mathcal O\left(\frac{\log\beta}{\beta}\right)\right]+\mathcal O(1)\\
&\leq \beta(1- c\delta_\lambda)\left[1+\mathcal O\left(\frac{\log\beta}{\beta}\right)\right]
,\quad  \hbox{in}\ \Omega^\pm\setminus \{|t|<\delta_\lambda\},
\end{aligned}
\]
as claimed.
%which implies
%\[
%\lambda^2e^{\, w_0^\pm}\leq \lambda^2 e^{\,\beta(1- c\delta_\lambda)}\left[1+\mathcal O\left(\frac{\log\beta}{\beta}\right)\right]\leq Ce^{\,- c\beta\delta_\lambda} \ \hbox{in}\  \Omega^\pm\setminus \mathcal U_{\delta_\lambda}. 
%\]
\end{proof}

Let $\chi(\eta)$ be a smooth cutoff function such that $\supp \chi \subset \{|\eta|\leq K\log \beta \}$, where $K\log \beta>4\lambda\mu_\lambda \delta_\lambda=4M\log\beta$ and $\chi(\eta)=1$ in the support of $\chi_0$.  Let $\chi^\pm(x)$ be  cutoff functions supported in the set $\Omega\setminus\supp \chi_0$.
We will split the calculations in (\ref{fix 14}) between the inner error $\chi N_\lambda(u_0)$ and the outer error $\chi^\pm N_\lambda(u_0)$.  In the inner part we express $N_\lambda$ in terms of $(\xi, \eta)$ and write $\chi{\tt N}_\lambda({\tt u_0})$. 
\begin{lemma}\label{lem fix 3}
Given $\theta<\min\{\omega^*, 1/2, 1/Q\}$ we have
\begin{align}
\label{fix 15}
\|\chi{\tt N}_\lambda({\tt u}_0)\|_{L^2_\theta(C_{\alpha|\gamma|})} &\lesssim \beta^{3/2+\theta K}\log^4\beta,\\
\label{fix 16}
\|\chi^\pm {N}_\lambda(u_0)\|_{H^1(\Omega^\pm)} &\lesssim \beta^{-100}. 
\end{align}
\end{lemma}
\begin{proof}
We need to estimate the four terms on the right hand side of (\ref{fix 14}) restricted to  the $\supp \chi$, denote these restrictions  by ${\tt E}_j$, $j=1, \dots, 4$.  
We will rely on the fact that ${\tt E}_j=0$ when $|\eta|>K\log\beta$. We have ${\tt E}_1={\tt N}_\lambda({\tt v}_0(\cdot; f))$ and then using (\ref{fix 10}) and (\ref{fix 1}) we get
\[
\|{\tt E}_1\|_{L^2_\theta(C_{\alpha|\gamma|})}\lesssim \beta^2(\|f\|_{H^2(\mathcal M)}+\beta^{-1/2+\theta K})\lesssim \beta^{3/2+\theta K}.
\]

To calculate the commutator terms we use the expression of the gradient $\nabla_x$  in the local coordinates $(s, t)$ and $(\xi, \eta)$. For any function ${g}={g}(x)$ defined locally near $\gamma$ we have in terms of the Fermi variables $(s,t)$
\[
\nabla_x g\ = \frac{1}{A} \partial_s g\partial_s+\partial_t g\partial_t,
\]
and setting ${\tt g}(\xi, \eta)=g(\xi/\alpha, \eta/\lambda\mu_\lambda)$ we get
\[
\nabla_x{g}=\frac{\alpha}{A}\partial_\xi{\tt g}\partial_\xi+\lambda\mu_\lambda\partial_\eta{\tt g}\partial_\eta=\frac{1}{A}\left.\left[\partial_s g-\frac{\eta}{\lambda\mu^2_\lambda}\partial_s\mu_\lambda\partial_t g\right]\partial_\xi+\partial_t g\partial_\eta\right|_{(s,t)=(\xi/\alpha, \eta/\lambda\mu_\lambda)}
\]
In the case at hand we have $\chi^\pm_0=\chi^\pm_0(t/\delta_\lambda)=\chi^\pm_0(\eta/M\log\beta)$ hence
\begin{equation}
\label{com 1}
\nabla_x \chi_0^\pm=\mathcal O(\delta^{-1}_\lambda)  (\chi^\pm_0)'(\eta/M\log\beta)\partial_\eta.
\end{equation}
Similarly, using (\ref{change variables}), we get
\begin{equation}
\Delta_x\chi^\pm_0=\mathcal O(\delta_\lambda^{-2})(\chi_0^\pm)''(\eta/M\log\beta)+\mathcal O(\delta^{-1}_\lambda)(\chi_0^\pm)'(\eta/M\log\beta).
\label{com 2}
\end{equation}
We write ${\tt v}_0(f)$ and ${\tt w^\pm_0}$ when we express $v_0(f)$ and $w^\pm_0$ in terms of $(\xi, \eta)$. Use \eqref{uzero 3} and the definition of $w_0^\pm$ together with \eqref{match 1b} and \eqref{match 1c}, to get in the $\supp \chi_0$
\[
\begin{aligned}
{\tt v}_0-{\tt w_0^\pm}&= b_0+2\ln\mu_\lambda-a_0\eta+\mathcal O\left(e^{-a_0(\eta+f)}\right)+\mathcal O(|f|+|\partial_\eta f|)\\
&-\left[b_0+2\ln\mu_\lambda-a_0\eta+\mathcal O\left(\frac{\eta}{\beta}\right)+\mathcal O\left(\frac{\log \beta}{\beta}\right)+\mathcal O\left(\frac{\eta^2}{\beta} \right)\right]\\
&=\mathcal O\left(e^{-a_0(\eta+f)}\right)+\mathcal O(|f|+|\partial_\eta f|)+\mathcal O\left(\frac{\log^2\beta}{\beta} \right)
\end{aligned}
\]
and 
\[
\lambda\mu_\lambda \partial_\eta({\tt v}_0-{\tt w}_0^\pm)=\mathcal O\left(\beta e^{-a_0(\eta+f)}\right)+\mathcal O\left(\beta(|f|+|\partial_\eta f|+|\partial_{\eta}^2 f|)\right)+\mathcal O(\log\beta).
\]
Based on the above one of the terms of the commutator can be estimated as follows
\[
\|\mathcal O(\delta^{-2}_\lambda)(\chi_0^\pm)''(\eta/M\log\beta) ({\tt v}_0-{\tt w_0^\pm})\|_{L^2_\theta(C_{\alpha|\gamma|})}\lesssim \beta^{3/2+\theta K},
\]
(to estimate exponentially decaying terms we take the constant $M$ in the definition of $\delta_\lambda$ large). Estimating similarly the remaining commutator terms we get
\[
\|{\tt E}_j\|_{L^2_\theta(C_{\alpha|\gamma|})}\lesssim \beta^{3/2+\theta K}, \quad j=2,3.
\]
Using Lemma \ref{lem fix 2}, the affine behavior of $U$ at $\pm \infty$  and choosing $M$ large we estimate $\|{\tt E}_4\|_{L^2_\theta(C_{\alpha|\gamma|})}$  in the same way.  This proves (\ref{fix 15}). 

Next we consider  (\ref{fix 16}). Because of the choice of the cutoff functions $\chi^\pm$ the only term that we need to estimate is the last member on the right hand side of (\ref{fix 14}). Choosing $M$ in (\ref{def del}) large enough and using Lemma \ref{lem fix 2} we get (\ref{fix 16}) as claimed.

\end{proof}

\subsection{Estimates of the nonlinear terms in the equation}

\subsubsection{Patching  the inner and the outer problem}\label{sec patch}

The approximate solution defined in (\ref{fix 7}) depends on the    inner functions $(f, \bar v_1, v_2)$ and the  outer functions $w_1^\pm$ patched together through some cutoff functions.  To solve our problem by a fixed point argument we will first split the equation $N_\lambda(u)=0$ into a system of coupled nonlinear equations for the inner and the outer part and then we will patch them back together. To do this we need to define additional cutoff functions $\rho$ and $\rho^\pm$ in the following way:
\begin{equation}
\label{fix 17}
\begin{aligned}
&\supp \rho\subset \{|t|\leq m_\rho\delta_\lambda\}, \qquad m_\rho>2m_1, \qquad \chi_{2}\rho=\rho, \quad \chi^\pm_1(1-\rho)=1-\rho,\\
&\supp\rho^\pm\subset \{|t|>\delta_\lambda/2\}\cap \Omega^\pm, \qquad \rho^\pm \chi_1^\pm =\chi_1^\pm.
\end{aligned}
\end{equation}
We will also need a cutoff function $\bar \rho$ such that  $\supp\bar \rho \subset\{|t|\leq m_{\bar \rho} \delta_\lambda\}$, with some $m_{\bar \rho}>m_2$, and 
$\bar\rho\chi_2=\chi_2$. 
Next we write
\[
N_\lambda(u_0+u_1+u_2)=N_\lambda (u_0)+ DN_\lambda(u_0)(u_1+u_2)+F_\lambda(u_0; u_1+u_2), 
\]
where 
\[
DN_\lambda(u_0)=\Delta +\lambda^2 e^{\,u_0}, \quad F_\lambda=N_\lambda(u_0+u_1+u_2)-N_\lambda (u_0)- DN_\lambda(u_1+u_2). 
\]
Let 
\[
L_{\xi, \eta}=(\partial_\eta^2+p_\alpha(\xi)\partial_\xi^2)+ e^{\,U(\eta)}
\]
be the linear operator considered in section \ref{lin theory} and let $L_x$ be the operator $(\lambda\mu_\lambda)^2 L_{\xi, \eta}$ transferred  to a small neighborhood of $\gamma$ by the diffeomorphism $X_{\gamma, \lambda}(x)=(\xi, \eta)$.  Recall that  ${\tt v}_2=v_2\circ X_{\gamma, \lambda}^{-1}$.  In the inner region we have
\begin{equation}
\label{dnl}
\begin{aligned}
\left[\left(DN_\lambda(u_0)-L_x\right)v_2\right]\circ X_{\gamma, \lambda}^{-1}(\xi, \eta) &=
\left[\Delta_{\xi, \eta}-(\lambda\mu_\lambda)^2(\partial_\eta^2+p_\alpha\partial_\xi^2)\right]{\tt v}_2\\
&\quad +\lambda^2 \left[\exp\left(\chi_0 {\tt v_0}(f)+\chi^+_0 {\tt w}^+_0+\chi^-_0 {\tt w}^-_0\right)-\exp\left({\tt v}_0(0)\right)\right] {\tt v}_2\\
&:= {\tt S}_\lambda ({\tt v}_2)+{\tt G}_\lambda({\tt v}_2).
\end{aligned}
\end{equation}
The first term above  carries the derivatives of ${\tt v}_2$, the second  depends linearly on ${\tt v}_2$ only. For later purpose we denote $S_\lambda={\tt S}_\lambda \circ X_{\gamma, \lambda}$, $G_\lambda={\tt G}_\lambda\circ X_{\gamma, \lambda}$. 

With all these notations the  inner equation is
\begin{equation}
\label{fix 18}
\begin{aligned}
-L_x v_2 &= \bar \rho \left[S_\lambda(v_2)+G_\lambda(v_2)\right]+\rho \left[N_\lambda(u_0)+F_\lambda\right]\\
&\quad +\left\{\chi_1 DN_\lambda(u_0)\bar v_1+\left[\Delta, \chi^+_1\right](w^+_1-\bar v_1)+ \left[\Delta, \chi^-_1\right](w^-_1-\bar v_1)\right\}
\end{aligned}
\end{equation}
and the outer equations are
\begin{equation}
\label{fix 19}
-\Delta w^\pm_1=(1-\rho)\left[N_\lambda (u_0)+F_\lambda\right]+\left[\Delta, \chi_{2}^{}\right]v_2+\lambda^2 \rho^\pm e^{\,u_0} w_1^\pm.
\end{equation}
Multiply the first equation by $\chi_{2}$ and the second by $\chi_1^\pm$, use that  
\[
\chi_2\bar\rho=\chi_2, \quad \chi_{2}\rho=\rho, \quad \chi^\pm_1(1-\rho)=1-\rho, 
 \quad \chi^\pm_1\rho^\pm=\chi_1^\pm, 
\]
and add the resulting equations to obtain $N_\lambda(u_0+u_1+u_2)=0$. 

Denote the right hand sides of (\ref{fix 18}) and (\ref{fix 19}) respectively by $\mathcal F$ and $\mathcal F^\pm$ so that 
\[
\mathcal F=\mathcal F(f, \bar v_1, v_2, w^+_1, w^-_1), \qquad \mathcal F^\pm =\mathcal F(f, \bar v_1, v_2, w^+_1, w^-_1).
\]
Suppose that $(\bar v_1, v_2, w^+_1, w^-_1)$ are given. The fix point argument consist  of determining functions $(\bar\phi_1,  \phi_2, \phi^+_1, \phi^-_1)$ such that
\begin{align}
\label{fix 20}
-L_x \phi_2 &=\mathcal F(f, \bar v_1, v_2, w^+_1, w^-_1)\\
\label{fix 21}
-\Delta \phi^\pm_1&=\mathcal F^\pm =\mathcal F^\pm(f, \bar v_1, v_2, w^+_1, w^-_1)
\end{align}
to define a map:
\[
(\bar v_1, v_2, w^+_1, w^-_1)\longmapsto (\bar\phi_1, \phi_2, \phi^+_1, \phi^-_1)
\]
and then showing that it has a fixed point. In the process of solving (\ref{fix 20}) we need to determine the modulation function $f$ in such a way that the set of orthogonality conditions on the right hand side of (\ref{fix 20}) is satisfied. Solving the modulation equation can be incorporated  into the fixed point scheme through Corollary \ref{cor fix 1}. 

\subsubsection{The matching condition} 

 Let us suppose that $\bar v_1$ expressed in $(\eta, s)$ is a function $\bar {\tt v}_1\in \mathcal \mathcal K_\gamma$ such that 
\[
\bar {\tt v}_1(\eta, \cdot)=h^\pm_1+h_2^\pm \eta+\mathcal O(e^{\,-a_0|\eta|}), \quad \eta\to \pm \infty.
\]
As in Proposition \ref{cru}  we should  require that 
\[
\begin{aligned}
w_1^\pm&=0, \qquad \mbox{on}\ \partial\Omega^\pm\cap\partial\Omega,\\
w_1^\pm&=h_1^\pm,\qquad \mbox{on}\ \gamma,\\
\partial_n w_1^\pm&={\lambda\mu_\lambda}h_2^\pm,  \qquad \mbox{on}\ \gamma.
\end{aligned}
\]
The matching condition stated in this form is not very convenient for the fixed point argument and we will give now an alternative statement, as promised at the end of section \ref{sec add}.  Let $m_1>0$ be the constant in the definition of  $\chi_1$ in (\ref{fix 5}) and let $M$ be the constant in the definition of $\delta_\lambda$. Let also $\theta<1-\frac{1}{m_1M}$ be given. Consider a function $w$ defined in a neighborhood $V$ of $\gamma$ expressed in terms of $(s,t)$. Taylor's expansion of $w$ is
\[
w(s,t)=w(s,0)+\partial_t w(s,0)+ t^2\int_0^1\int_0^\sigma \partial_t^2 w(s,\tau t)\,d\tau d\sigma.
\]
Let us suppose that  another function $\bar v(s,t)$ is given in such a way that 
\[
w(s,t)-\bar v(s,t)=t^2\int_0^1\int_0^\sigma \partial_t^2 w(s,\tau t)\,d\tau d\sigma+\tilde v(s,t), 
\]
where $\tilde v(s,t)=h(s) \mathcal O(e^{\,-2 \lambda \mu_\lambda t})$, and $h\in L^2(0,|\gamma|)$. 
Let us denote
\[
{\tt K}^\pm(\xi, \eta)=\left[\Delta_x, \chi^\pm_1\right](w-\bar v)\circ X^{-1}_{\gamma, \lambda}(\xi, \eta).
\]
We state now:
\begin{lemma}\label{lem patch}
Under the above conditions and notations it holds
\begin{equation}
\label{patch 0}
\|{\tt K}^\pm\|_{L^2_\theta(C_{\alpha|\gamma|})}\lesssim \beta^{1+2\theta m_1M}\left(\|w\|_{H^2(V)}+\|h\|_{L^2(0, |\gamma|)}\right).
\end{equation}
\end{lemma}
\begin{proof}
We have
\[
{\tt K}^\pm=\left((\Delta_x\chi^\pm_1)( w-\bar v)\right)\circ X^{-1}_{\gamma, \lambda}(\xi, \eta)+\left(\nabla_x\chi^\pm_1)\nabla_x( w-\bar v)\right)\circ X^{-1}_{\gamma, \lambda}(\xi, \eta).
\]
We use (\ref{com 1})--(\ref{com 2}) with $\chi=\chi_1^\pm$ 
\[
\begin{aligned}
{\tt K}^\pm&=\mathcal O(\delta_\lambda^{-2})(\chi_1^\pm)''(\eta/M\log\beta) ( w-\bar v)\circ X^{-1}_{\gamma, \lambda}(\xi, \eta)+\mathcal O(\delta^{-1}_\lambda)(\chi_1^\pm)'(\eta/M\log\beta)( w-\bar v)\circ X^{-1}_{\gamma, \lambda}(\xi, \eta)\\
&\quad +\left[\mathcal O(\delta^{-1}_\lambda)  (\chi^\pm_0)'(\eta/M\log\beta)\partial_\eta\right]\nabla_x( w-\bar v)\circ X^{-1}_{\gamma, \lambda}(\xi, \eta)\\
&=\sum_{j=1}^3 {\tt K}^\pm_j.
\end{aligned}
\]
We will calculate $\|{\tt K}^\pm_1\|_{L^2_\theta(C_{\alpha|\gamma|})}$. Set 
\[
{\tt R}(\xi, \eta)=\frac{\eta}{\lambda\mu_\lambda} \int_0^1\int_0^\sigma \partial_t^2 w(\xi/\alpha,\tau \eta/\lambda\mu_\lambda)\,d\tau d\sigma= \int_0^1
\int_0^{\sigma \eta/\lambda\mu_\lambda}\partial_t^2 w(\xi/\alpha,t)\,dt d\sigma.
\]
We have
\[
\begin{aligned}
%\|{\tt K}^\pm_1\|_{L^2_\theta(C_{\alpha|\gamma|})}&\leq 
&\delta_\lambda^{-2}\left\{\int_\R\!\int_0^{\alpha|\gamma|} |(\chi_1^\pm)''(\frac{\eta}{M\log\beta})|^2\left(\frac{\eta}{\lambda\mu_\lambda}\right)^2 e^{\,2\theta|\eta|}\ {\tt R(\xi, \eta)}^2\,d\xi d\eta\right\}^{1/2}
\\
&\lesssim \delta_\lambda^{-2}\left\{\int_\R\int_0^{\alpha|\gamma|}|(\chi_1^\pm)''(\frac{\eta}{M\log\beta})|^2\left(\frac{\eta}{\lambda\mu_\lambda}\right)^3e^{\,2\theta|\eta|}\left(\int_0^{\alpha|\gamma|}\!\int_0^{\eta/\lambda\mu_\lambda} |\partial_t^2 w(\xi, t)|^2\,d\eta d\xi\right)
\right\}^{1/2}\\
&\lesssim  \delta_\lambda^{-2} \beta^{-1} \beta^{2\theta mM}\|w\|_{H^2(V)}\\
&\lesssim \beta^{1+2\theta mM} \|w\|_{H^2(V)}.
\end{aligned}
\]
We also have
\[
\begin{aligned}
&\delta_\lambda^{-2}\left\{\int_\R\!\int_0^{\alpha|\gamma|} |(\chi_1^\pm)''(\frac{\eta}{M\log\beta})|^2 e^{\,2\theta|\eta|}|\tilde v(\xi/\alpha, \eta/\lambda\mu_\lambda)|^2\,d\xi d\eta
\right\}^{1/2} \\
&\lesssim \delta_\lambda^{-2}\beta^{1/2} \|h\|_{L^2(\gamma)}\left\{\int_\R |(\chi_1^\pm)''(\frac{\eta}{M\log\beta})|^2 e^{\,-2(1-\theta)|\eta|}\,d\eta\right\}^{1/2}
\\
&\lesssim \beta^{5/2} \beta^{-(1-\theta) m_1M} \|h\|_{L^2(0,|\gamma|)}
\\
&\lesssim \beta^{1+2\theta m_1M} \|h\|_{L^2(0,|\gamma|)}
\end{aligned}
\]
where the last inequality follows from the choice of $\theta$. We have found
\begin{equation}
\label{patch 1}
\|{\tt K}^\pm_1\|_{L^2(C_{\alpha|\gamma|})}\lesssim \beta^{1+2\theta m_1M}\left(\|w\|_{H^2(V)}+\|h\|_{L^2(0,|\gamma|)}\right)
\end{equation}
Estimates of ${\tt K}^\pm_j$, $j=2,3$ are similar and lead to the analog of  (\ref{patch 1}). This ends the proof.
\end{proof}

Based on the results of the above Lemma (which applies with $w$ replaced by $w^\pm$) we will require that the following {matching condition} is satisfied 
\begin{equation}
\label{fix 22}
\left\|\left[\Delta_x, \chi^\pm_1\right](w^\pm_1-\bar v_1)\circ X^{-1}_{\gamma, \lambda}\right\|_{L^2_\theta(C_{\alpha|\gamma|})}\lesssim \beta^{1+2\theta m_1M}\left(\|w_1^\pm\|_{L^2(\Omega^\pm)}+\|\bar {\tt v}_1\|_{L^2(\mathcal K_\gamma)}\right)
\end{equation}
Conditions (\ref{fix 1})--(\ref{fix 3}) together with (\ref{fix 22}) form the set of {\it a priori} assumption on the unknowns of the problem and the solution of (\ref{fix 20})--(\ref{fix 21}) should,  consistently, satisfy the same conditions.  
 
\subsubsection{Estimate of the right hand side of (\ref{fix 18})} 

We will express $\mathcal F$ in (\ref{fix 20}) in the local, stretched variables $(\xi, \eta)$. By Lemma \ref{lem fix 3} we know
\begin{equation}
\label{fix 23}
\|\rho {\tt N}_\lambda({\tt u}_0(\cdot, f)\|_{L^2_\theta(C_{\alpha|\gamma|})}\lesssim \beta^{3/2+\theta m_\rho M}\log^4\beta,
\end{equation}
where $m_\rho$ is the constant in the definition of the cutoff function $\rho$, see (\ref{fix 17}).  
We need to estimate the remaining terms 
\[
\begin{aligned}
E_{int}&=\bar \rho \left[S_\lambda(v_2)+G_\lambda(v_2)\right]+\rho F_\lambda
+\chi_1 DN_\lambda(u_0)\bar v_1+\left[\Delta, \chi^+_1\right](w^+_1-\bar v_1)+ \left[\Delta, \chi^-_1\right](w^-_1-\bar v_1)\\
&=:\sum_{j=1}^5 E_j.
\end{aligned}
\]
Let ${\tt E}_{int}={\tt E}_j(\xi,\eta)=E_j\circ X^{-1}_{\gamma, \lambda}(\xi, \eta)$.
\begin{lemma}\label{lem fix 4}
Let $m_1$ be the constant in the definition of $\chi_1$ and $M$ be the constant in the definition of $\delta_\lambda$. Suppose that the functions $(f, \bar {v}_1, \tilde v_1, v_2, w^\pm_1)$ satisfy {\it a priori} hypothesis  (\ref{fix 1})--(\ref{fix 3}) and  the matching condition (\ref{fix 22}). There exists 
$0<\bar\theta<\min\{\omega^*, 1/2,1/Q\}$ such that for all  $M$ large enough and all $\theta$ such that
\begin{equation}
\label{cond theta}
\theta<\min\left\{\bar \theta, 1-\frac1{m_1M}\right\},
\end{equation}
it holds
\begin{equation}
\begin{aligned}
\label{est int}
\|{\tt E}_{int}\|_{L^2_\theta(C_{\alpha|\gamma|})}&\lesssim \beta\|{\tt v}_2\|_{H^2_\theta(C_{\alpha|\gamma|})} (1+\beta\|f\|_{H^2(\mathcal M)})
\\
&\quad  + \beta^{5/2} \left(\|\bar {\tt v}_1\|^2_{L^2(K_\gamma)}+\|w_1^+\|^2_{H^2(\Omega^+)}
+\|w_1^-\|^2_{H^2(\Omega^+)}\right)+\beta^2\|{\tt v}_2\|^2_{H^2_\theta(C_{\alpha|\gamma|})}\\
&\quad  +\left(\beta^{1/2+4\theta m_1M} \|\bar{\tt v}_1\|_{H^2(\mathcal K_\gamma)}+\beta^{3/2+4m_1\theta M} \|\bar{\tt v}_1\|_{L^2(\mathcal K_\gamma)}\right)\\
&\quad +\beta \|\bar{\tt v}_1\|_{H^2(\mathcal K_\gamma)}(1+\beta\|f\|_{H^2(\mathcal M)}) \\
&\quad + \beta^{1+2\theta m_1M}\left(\|w_1^+\|_{L^2(\Omega^\pm)}+\|w_1^-\|_{L^2(\Omega^\pm)}+\|\bar {\tt v}_1\|_{L^2(\mathcal K_\gamma)}\right).
\end{aligned}
\end{equation}
\end{lemma}
\begin{proof}
%Denote $v=\tilde v_1+v_2$ and ${\tt v}(\xi,\eta)=v\circ X^{-1}_{\gamma, \lambda}(\xi, \eta)$. 
We have, using (\ref{dnl})
\[
\begin{aligned}
{\tt E}_1&= \bar \rho\left[\Delta_{\xi, \eta}-(\lambda\mu_\lambda)^2(\partial_\eta^2+p_\alpha\partial_\xi^2)\right]{\tt v}+\lambda^2\bar \rho \left[\exp\left(\chi_0 {\tt v_0}(f)+\chi^+_0 {\tt w}^+_0+\chi^-_0 {\tt w}^-_0\right)-\exp\left({\tt v}_0(0)\right)\right] {\tt v}\\
&={\tt E}_{11}+{\tt E}_{12},
\end{aligned}
\]
where, we recall,  ${\tt v}_0(f)=U(\eta+f)+2\log\left[\mu_\lambda(1+\partial_\eta f)\right]$. To estimate ${\tt E}_{11}$ we use  (\ref{fix 9})
\[
\|{\tt E}_{11}\|_{L^2_\theta(C_{\alpha|\gamma|})}\lesssim \beta\|{\tt v}\|_{H^2_\theta(C_{\alpha|\gamma|})}.
\]
To estimate ${\tt E}_{12}$ we use (\ref{fix 1}), Lemma \ref{lem fix 2} and the asymptotic  expansion of $U$ (\ref{u-asy}).
We write 
\[
\begin{aligned}
{\tt E}_{12} &=\lambda^2\bar \rho \left[\exp\left(\chi_0 {\tt v_0}(f)+\chi^+_0 {\tt w}^+_0+\chi^-_0 {\tt w}^-_0\right)-\exp\left({\tt v}_0(f)\right)\right] {\tt v}_2+\lambda^2\rho \left[\exp\left({\tt v}_0(f)\right)-\exp\left({\tt v}_0(0)\right)\right]{\tt v}_2\\
&={\tt E}_{121}+{\tt E}_{122}.
\end{aligned}
\]
Since $\chi_1=1$ when $|\eta|\leq \delta_\lambda\lambda\mu_\lambda$,
\[
\|{\tt E}_{121}\|_{L^2_\theta(C_{\alpha|\gamma|})}\lesssim \beta^2 e^{-cM\log\beta}\|{\tt v}_2\|_{H^2_\theta(C_{\alpha|\gamma|})}\lesssim  \beta\|{\tt v}_2\|_{H^2_\theta(C_{\alpha|\gamma|})},
\]
choosing $M$ in the definition of $\delta_\lambda$ large enough. Using the fact that if ${\tt v}_2\in H^2_\theta(C_{\alpha|\gamma|})$ then ${\tt v}_2\in H^2(C_{\alpha|\gamma|})$ and the Sobolev embedding we get, taking $\theta<1$
\[
\|{\tt E}_{121}\|_{L^2_\theta(C_{\alpha|\gamma|})}\lesssim \beta^2\|{\tt v}_2\|_{L^\infty(C_{\alpha|\gamma|})}\|f\|_{L^\infty(\mathcal M)}\lesssim \beta^2 \|{\tt v}_2\|_{H^2_\theta(C_{\alpha|\gamma|})}\|f\|_{H^2(\mathcal M)},
\]
so that 
\begin{equation}
\label{fix 24}
\begin{aligned}
\|{\tt E}_{1}\|_{L^2_\theta(C_{\alpha|\gamma|})}&\lesssim \beta \|{\tt v}_2\|_{H^2_\theta(C_{\alpha|\gamma|})}(1+\beta\|f\|_{H^2(\mathcal M)}).
\end{aligned}
\end{equation}

To estimate ${\tt E}_2$ we note that explicitly we have
\[
F_\lambda (u_0; u_1+u_2)=\lambda^2 e^{\,u_0}\left[e^{\, u_1+u_2}-1-(u_1+u_2)\right].
\]
We claim that there exists $\bar \theta>0$ such that 
\begin{equation}
\label{fix 25}
\lambda^2|\rho e^{\,{\tt u}_0}|\lesssim \beta^2 e^{\,-\bar\theta|\eta|}.
\end{equation}
This claim can be proven using the asymptotic behavior of $U$ and Lemma \ref{lem fix 2}. Writing ${\tt F}_\lambda(\xi,\eta)=F_\lambda\circ X^{-1}_{\gamma, \lambda}(\xi, \eta)$ we have
\[
|\rho {\tt F}_\lambda|\lesssim \beta^2 e^{\,-\bar\theta|\eta|}\left(|\chi_1\bar {\tt v}_1|^2+|\rho\chi^+_1{\tt w}^+_1|^2+|\rho\chi^-_1{\tt w}^-_1|^2
+|\chi_2{\tt v}_2|^2\right).
\]
We take $\theta<\bar \theta$ to get 
\begin{equation}
\label{fix 26}
\|\rho {\tt F}_\lambda\|_{L^2_\theta(C_{\alpha|\gamma|})}\lesssim \beta^{5/2} \left(\|\bar {\tt v}_1\|^2_{L^2(K_\gamma)}+\|w_1^+\|^2_{H^2(\Omega^+)}
+\|w_1^-\|^2_{H^2(\Omega^+)})\right)+\beta^2\|{\tt v}_2\|^2_{H^2_\theta(C_{\alpha|\gamma|})}.
\end{equation}

Next we estimate $E_3$. We have in terms of $(\xi, \eta)$:
\[
\begin{aligned}
{\tt E}_3 &=\chi_1(\lambda\mu_\lambda)^2p_\alpha\partial_\xi^2 \bar {\tt v}_1-\chi_1\lambda\mu_\lambda\varkappa\partial_\eta \bar{\tt v}_1+(\lambda\mu_\lambda)^2 \chi_1\left(\tilde a_{20}\partial_\eta^2+\tilde a_{11}\partial_{\xi, \eta}+\tilde a_{02}\partial_\xi^2+\tilde b_1\partial_\eta+\tilde b_2\partial_\xi\right)\bar {\tt v}_1\\
&\quad + \lambda^2\chi_1 \left(e^{\, {\tt u}_0(f)}-e^{\,{\tt v}_0(0)}\right)\bar{\tt v}_1
\end{aligned}
\]
The first two members above are estimated as follows
\[
\begin{aligned}
\|\chi_1(\lambda\mu_\lambda)^2 p_\alpha \partial_\xi^2\bar {\tt v}_1\|_{L^2_\theta(C_{\alpha|\gamma|})}& \lesssim \beta^{1/2+4m_1M\theta} \|\bar{\tt v}_1\|_{H^2(\mathcal K_\gamma)}\\
\|\chi_1\lambda\mu_\lambda\varkappa\partial_\eta\bar {\tt v}_1\|_{L^2_\theta(C_{\alpha|\gamma|})}& \lesssim \beta^{3/2+4m_1M\theta} \|\bar{\tt v}_1\|_{L^2(\mathcal K_\gamma)},
\end{aligned}
\]
where $m_1$ is the constant in the definition of $\chi_1$ and $M$ is the constant in the definition of $\delta_\lambda$. The third term is estimated similarly and it is smaller. The last term is estimated in the same way as ${\tt E}_{12}$ above. As a consequence we get
\begin{equation}
\label{fix 27}
\|{\tt E}_3\|_{L^2_\theta(C_{\alpha|\gamma|})}\lesssim\left(\beta^{1/2+4m_1M\theta} \|\bar{\tt v}_1\|_{H^2(\mathcal K_\gamma)}+\beta^{3/2+4m_1M\theta} \|\bar{\tt v}_1\|_{L^2(\mathcal K_\gamma)}+\beta \|\bar{\tt v}_1\|_{H^2(\mathcal K_\gamma)}(1+\beta\|f\|_{H^2(\mathcal M)}) \right)
\end{equation}

For  $E_j$ and $j=4,5$ we have directly by the matching condition (\ref{fix 22})
\begin{equation}
\label{fix 28}
\|{\tt E}_4\|_{L^2_\theta(C_{\alpha|\gamma|})}+\|{\tt E}_5\|_{L^2_\theta(C_{\alpha|\gamma|})}\lesssim \beta^{1+2\theta m_1M}\left(\|w_1^+\|_{L^2(\Omega^\pm)}+\|w_1^-\|_{L^2(\Omega^\pm)}+\|\bar {\tt v}_1\|_{L^2(\mathcal K_\gamma)}\right).
\end{equation}
Combining (\ref{fix 24}), (\ref{fix 26}), (\ref{fix 27}) and (\ref{fix 28}) we obtain the assertion of the Lemma. 

\end{proof}

\subsubsection{Estimate of the right hand side of (\ref{fix 19})}
By (\ref{fix 16}) we have
\[
\|(1-\rho) N_\lambda (u_0)\|_{L^2(\Omega^\pm)}\lesssim  \beta^{-100}.
\]
We will now estimate the remaining terms in outer error
\[
\begin{aligned}
E^\pm_{out}&=(1-\rho) F_\lambda+\left[\Delta, \chi_2^{}\right]v_2+\lambda^2 \rho^\pm e^{\,u_0} w_1^\pm=:\sum_{j=1}^3 E_j^\pm,
\end{aligned}
\]
see (\ref{fix 19}). 

\begin{lemma}\label{lem fix 5}
Let $\theta>0$ be given, let $c>0$ be the constant in the statement of  Lemma \ref{lem fix 2} and let $M$ be the constant in the definition of $\delta_\lambda$. Suppose that the functions $(f,\bar {v}_1, v_2, w^\pm_1)$ satisfy {\it a priori} hypothesis  (\ref{fix 1})--(\ref{fix 3}) and  the matching condition (\ref{fix 22}). Then the constant $m_1$ in the definition of $\chi_1$,  the constant $m_\rho$ in the definition of $\rho$ and the constant $m_2$ in the definition of $\chi_2$ can be chosen in such a way that
\begin{equation}
\label{err out 1}
\|E^\pm_{out}\|_{L^2(\Omega^\pm)}\lesssim \beta^{-3/2}\left(\|w^\pm_1\|_{H^2(\Omega^\pm)}+\|\bar {\tt v}_1\|_{H^2(\mathcal K_\gamma)}+\|{\tt v}_2\|_{H^2_\theta(C_{\alpha|\gamma|})}
\right),
\end{equation} 
for all large enough $\beta$. 
\end{lemma}
\begin{proof}
We begin with $E_1^\pm$. 
We have, using the asymptotic  properties of $U$ and Lemma \ref{lem fix 2}
\[
|(1-\rho) F_\lambda| \lesssim \beta^2 e^{\,-c m_\rho  \delta_\lambda\lambda\mu_\lambda} (1-\rho)\left(|\chi_1\bar v_1|^2+|\chi^\pm_1w^\pm_1|^2+|\chi_{2} v_2|^2\right)
\]
Using the definition of $\delta_\lambda$ we get from this
\begin{equation}
\begin{aligned}
\label{fix 29}
\|E_1^\pm\|_{L^2(\Omega^\pm)} &\lesssim \beta^{2-cm_\rho M} \left(\|w_1^\pm\|^2_{H^2(\Omega^\pm)}+\|{\tt v}_2\|^2_{H^2_\theta(C_{\alpha|\gamma|})}+\|\bar {\tt v}_1\|^2_{H^2(\mathcal K_\gamma)}\right)\\
&<\beta^{-2} \left(\|w_1^\pm\|^2_{H^2(\Omega^\pm)}+\|{\tt v}_2\|^2_{H^2_\theta(C_{\alpha|\gamma|})}+\|\bar {\tt v}_1\|^2_{H^2(\mathcal K_\gamma)}\right),
\end{aligned}
\end{equation}
if $cm_\rho M>6$ and $\beta$ is large. 

To estimate $E_2^\pm$ we start with 
\[
\begin{aligned}
\|(\Delta \chi_{2}) v_2\|_{L^2(\Omega^+)} &\lesssim \beta^2 \left[\int_0^{|\gamma|}\!\int_{m_2\delta_\lambda}^{2m_2\delta_\lambda}|v_2|^2\,dtds\right]^{1/2}\\
&\lesssim \beta\left[\int_0^{\alpha|\gamma|}\!\int_{m_2M\log\beta}^{2m_2 M\log\beta} |{\tt v}_2|^2\,d\eta d\xi\right]^{1/2}\\
&\lesssim  e^{\,-\theta m_2M\log\beta}\|{\tt v}_2\|_{H^2_\theta(C_{\alpha|\gamma|})}.
\end{aligned}
\]
Estimates of other terms involved in $E_3^\pm$ are similar so that we get at the end
\begin{equation}
\label{fix 30}
\|E^\pm_3\|_{L^2(\Omega^\pm)}\lesssim  e^{\,-\theta m_2 M\log \beta}\|{\tt v}_2\|_{H^2_\theta(C_{\alpha|\gamma|})}
< \beta^{-2} \|{\tt v}_2\|_{H^2_\theta(C_{\alpha|\gamma|})},
\end{equation}
whenever $m_2$ is large so that $\theta m_2 M>6$.

We turn finally to $E^\pm_4$. By the definitions of $\rho^\pm$ and $\delta_\lambda$, the asymptotic  properties of $U$ and Lemma \ref{lem fix 2} we have
\[
|\lambda^2 \rho^\pm e^{\,u_0}|\lesssim \beta^2 e^{\,-\frac{1}{2}cm_1M\log\beta},
\]
hence
\begin{equation}
\label{fix 31}
\|E^\pm_4\|_{L^2(\Omega^\pm)}\lesssim \beta^{2-\frac{1}{2}cm_1 M}\|w^\pm_1\|_{H^2(\Omega^\pm)}<\beta^{-2}\|w^\pm_1\|_{H^2(\Omega^\pm)},
\end{equation}
whenever $cm_1 M>10$. 

Examining (\ref{fix 29})--(\ref{fix 31}) we see that it suffices to chose $m_1$, $m_\rho$, $m_2>m_\rho$  such that 
\begin{equation}
\label{fix 32}
\frac{1}{2} m_\rho>m_1>\frac{10}{cM}, \qquad m_2>\frac{6}{M\theta}.
\end{equation}
This ends the proof.
\end{proof}
The proof of the following corollary is similar.
\begin{corollary}
\label{cor err out}
Under the same assumption as in Lemma \ref{lem fix 5} it holds
\begin{equation}
\label{err out  2}
\|E^\pm_{out}\|_{H^1(\Omega^\pm)} \lesssim \beta^{-1/2}\left(\|w^\pm_1\|_{H^2(\Omega^\pm)}+\|\bar {\tt v}_1\|_{H^2(\mathcal K_\gamma)}+\|{\tt v}_2\|_{H^2_\theta(C_{\alpha|\gamma|})}
\right).
\end{equation}
\end{corollary}

\subsection{Solution of the fixed point problem}

\subsubsection{The adjustment of  parameters}

As we saw in the statements of Lemma \ref{lem fix 4} and \ref{lem fix 5} patching the inner and the outer solution depends on a somewhat intricate choice of the exponent $\theta$, the constant $M$ in the definition of $\delta_\lambda$ and the  constants $m_1, m_\rho$ in the definitions of, respectively,  $\chi_1,\rho$.  Let us summarize the restrictions we have made so far. From Lemma \ref{lem fix 4} we have that $M>\bar M$, where $\bar M$ is chosen large and $\theta<\min\{\bar\theta, 1-\frac{1}{m_1M}\}$, where $\bar\theta>0$ is fixed. At the same time $m_1, m_\rho, m_2$  need to satisfy (\ref{fix 32}). We see that  choosing $M$ larger and $\theta$ smaller if necessary  we will have to adjusts $m_1, m_\rho, m_2$ at the same time but all this can be done without any contradiction. 

The choice of $\bar \theta$ is associated with the choice of $\omega^*$  in Lemma \ref{lem fix 1} as we need $\bar \theta<\min\{\omega^*, \frac{1}{2},1/Q\}$. The relation between $\theta$ and $\omega$ is seen also in Lemma \ref{lem 4}. This is especially important in solving the modulation equation in the next section. Recall that we need to satisfy the  orthogonality conditions for the right hand side of this equation for all  $\omega_{\alpha, k}<\theta+\epsilon$, where $\epsilon$ is a small number independent on all other constants. Since at the same time $\omega_{\alpha, k}<\omega^*$ we need $\theta+\epsilon <\omega^*$ as well.   

We will also need to choose $\beta$ large and this means also $\alpha$ large and $\lambda$ small. Because the matching problem is resonant, see the statement of  Proposition \ref{cru},  we will have to be careful to avoid the resonant sequence of the parameters (up to now in this section this was not an issue). Let us recall that as long as $\alpha$ is such that 
\begin{equation}
\label{fix 33}
\min_{n>N_q} |\alpha_n-\alpha|>\frac{1}{4R},
\end{equation}
the assertion of Proposition \ref{cru} holds. But 
\[
\alpha= \frac{(\beta+2\log\beta)}{a_0\sqrt{R_1 R_2}\log\sqrt{\frac{R_1}{R_2}}}\qquad \mbox{where} \qquad \beta=2\log\frac{1}{a_0\lambda}+b_0,
\]
see (\ref{def alpha}). Because the relations defining $\alpha$ as a function of $\beta$ and $\beta$ as a function of $\lambda$ are  themselves invertible, continuous  functions from this we conclude the following:
\begin{corollary}\label{cor fix 2}
There exists an open set $\Lambda\subset (0,1)$ of $\lambda$ such that $0$ is in the closure of $\Lambda$ and for any $\lambda\in \Lambda$ the function  
$\alpha(\lambda)$ satisfies (\ref{fix 33}) and the assertion of Proposition \ref{cru} holds. 
\end{corollary} 
This is translated directly to a similar statement for the parameter $\beta$. We conclude that there exists an open, unbounded  set $B\in (1, \infty)$ such that whenever $\beta\in B$ then $\lambda\in \Lambda$. As we prefer to state our results in terms of $\beta$ rather than $\lambda$ in all of the following statements we will suppose, without saying it explicitly, that $\beta \in B$ is taken to be sufficiently large.

Finally, we note that in (\ref{fix 1})--(\ref{fix 3}) and  the matching condition (\ref{fix 22}) we introduced a constant $\sigma\in (0,1/4)$ which up to now has not be specified. In fact except for the assertion  of  Lemma \ref{lem fix 5}, where we used the size of the norms involved to control the quadratic terms, we have  not so far used the hypothesis on the  size of functions involved in (\ref{fix 1})--(\ref{fix 3}),  (\ref{fix 22}) but only that they belong to the corresponding functional spaces. The constant $\sigma$ will be chosen later on but at this point we mention that it is related to (\ref{fix 23}). With $\sigma\in (0, 1/4)$ fixed we will need 
\[
\theta<\frac{\sigma}{m_\rho M}
\]
where $m_\rho$ is the constant in the definition of the cutoff function $\rho$, see (\ref{fix 17}). This is another restriction on the size of $\theta$ which we will eventually impose. 

\subsubsection{The modulation parameter $f$ as a function of the unknowns} 

The first step in solving (\ref{fix 20}) is to determine the modulation function $f$ in such a way that a set of orthogonality conditions is satisfied, according to the theory developed in 
section \ref{lin theory}, see Lemma  \ref{lem 4}. We recall  the definition of the modulation function $f$ in (\ref{def ab})--(\ref{def ff}) and the definition of $K_\alpha$ in 
(\ref{asymp k}). In what follows we will solve the modulation equation for the number of modes $\tilde K_\alpha\leq  K_\alpha$, determined by the condition 
\[
\omega_{\alpha, k}\leq \theta+\epsilon< \omega^*, \qquad k=0,\dots, \tilde K_\alpha.
\]
Replacing $K_\alpha$ by $\tilde K_\alpha\leq K_\alpha$ does not change the theory of section \ref{sec modulation}. 

\begin{lemma}
\label{lem fix 6}
Given  $\sigma\in (0,1/4)$ as in (\ref{fix 2}) let $\theta<\min\{\bar\theta, 1-\frac{1}{m_1M}, \frac{1+\sigma}{4m_1M}\}$ and $\epsilon>0$ small be fixed so that $\theta+\epsilon<\omega^*$. Suppose that the functions $(\bar v_1, v_2, w^+_1, w^-_1)$ satisfy the hypothesis (\ref{fix 2})--(\ref{fix 3}). Then the modulation function $f$ can be determined in such a way that (\ref{fix 1}) is satisfied with $q=3$ and 
for  $0\leq k\leq \tilde K_\alpha$ it holds
\begin{equation}
\label{fix 34}
\int_0^{\alpha|\gamma|} \int_\R \frac{1}{(\lambda\mu_\lambda)^2}\left(\rho {\tt N}_\lambda({\tt u}_0(f))+{\tt E}_{int}\right)(\xi, \eta) k_{\omega_{\alpha, k}}(\eta)y_{\alpha, k}(\xi)\,d\eta d\xi=0, 
\end{equation}  
where $k_{\omega_{\alpha, k}}=k^+_{\omega_{\alpha, k}}$ and $k_{\omega_{\alpha, k}}=l_{\omega_{\alpha, k}}$ when $0\leq \omega_{\alpha, k}\leq \epsilon$ and
$k_{\omega_{\alpha, k}}=k^\pm_{\omega_{\alpha, k}}$ when $\epsilon<\omega_{\alpha, k}\leq \theta+\epsilon$. 
\end{lemma}
\begin{proof}
Recall that 
\[
f(\xi, \eta)= \left [a(\xi)\eta+ b(\xi)\right]Y(\eta), \quad Y(\eta)=\tanh\eta\sech\left(\frac{\eta}{Q}\right),
\]
with $a, b$ given in (\ref{def ab}) by their Fourier expansions in terms of $y_{\alpha, k}$  up to mode $\tilde K_\alpha$. 
Using (\ref{fix 10})--(\ref{fix 11}) we write the integrand in (\ref{fix 34}) in the form
\begin{equation}
\label{modulate 1}
\frac{1}{(\lambda\mu_\lambda)^2}\left(\rho {\tt N}_\lambda({\tt u}_0(f))+{\tt E}_{int}\right)(\xi, \eta)=a(\xi) {\tt Z}_1(\eta)+ p_\alpha(\xi)\partial_\xi^2 a(\xi) {\tt W}_1(\eta)+ b(\xi) {\tt Z}_2(\eta)+p_\alpha(\xi)\partial_\xi^2 b(\xi) {\tt W}_2(\eta)-{\tt h}(f)
\end{equation}
Noting that the integrand in (\ref{fix 34}) is compactly supported we can integrate against $k_{\omega_{\alpha, k}}y_{\alpha, k}$ as in (\ref{fix 34}).  Using  Lemma \ref{lem fix 1} and Corollary \ref{cor fix 1} we get
\begin{equation}
\label{fix 34.5}
\left({Z}_{\omega_{\alpha, k}} -\omega^2_{\alpha, k} {W}_{\omega_{\alpha, k}}\right)(a_k, b_k)^T=({\tt h}_{1k}, {\tt h}_{2k})=({\tt h}_{1,k}(f), {\tt h_{2k}}(f))
\end{equation}
where ${\tt h}_{jk}$ are defined as in (\ref{c lem fix 2}). To finish the proof we need to show that if $f$ satisfies (\ref{fix 1}) then the map $f\mapsto {\tt h}(f)$ determines a Banach contraction in the modulation space $\mathcal M$. Motivated by (\ref{size f}) we take $q=3$ in (\ref{fix 1}) and our goal is to show that 
\begin{equation}
\label{fix 35}
\|{\tt h}(f)\|_{L^2_\theta(C_{\alpha|\gamma|})}\lesssim \frac{\log^2\beta}{\sqrt{\beta}},
\end{equation}
assuming that $f$ satisfies (\ref{fix 1}). 

Consider  (\ref{fix 10}). Note that in the definitions of the functions ${\tt Z}_j$, ${\tt W}_j$, $j=1,2$ we take $U'=U'(\eta)$, while in $U', U''$ are taken as functions of $\eta+f$. This produces an term of size $\mathcal O(\|f\|^2_{H^2(\mathcal M)})$ in ${\tt h}$. The remaining terms in ((\ref{fix 10}), contained in lines two and three, result in the error of similar size. Thus the main in (\ref{fix 10}) is $\frac{U'\varkappa}{\lambda\mu_\lambda}$, and the size of its projection is estimated in (\ref{size f}). This means that the norm of the part of ${\tt h}$ that comes from (\ref{fix 10}) is consistent with (\ref{fix 35}). 

The remaining part of ${\tt h}$ comes from ${\tt E}_{int}$, whose size is controlled by (\ref{est int}). It is not hard to check that under the hypothesis (\ref{fix 1})--(\ref{fix 3}) this part of ${\tt h}$ also is estimated consistently with  (\ref{fix 35}) provided that 
\[
\sigma<\frac{1}{4}, \qquad \theta<\frac{1+\sigma}{4m_1M},
\]
as we assumed. In summary,  we have verified that the system (\ref{fix 34.5}) defines a map from the ball $\|f\|_{H^2(\mathcal M)}<\beta^{-1/2}\log^3\beta$ into itself. Checking that this map is a contraction is standard and we omit the details. 
\end{proof}

For the purpose of the fixed point argument which we will develop below we will rephrase the modulation equation slightly. Let us go back to the formula (\ref{modulate 1}) and its projections  (\ref{fix 34.5}) and note that the function ${\tt h}$ can be interpreted as a nonlinear function of $f$ and also of all other unknown functions $(\bar v_1, v_2, w^+_1, w^-_1)$. Then, with these functions as data we look for $\tilde f \in \mathcal K_\gamma$, 
\begin{equation}
\label{modulation 2a}
\tilde f (\xi, \eta)= [\tilde a(\xi)\eta+\tilde b(\xi)] Y(\eta),\quad \tilde a=\sum_{k=0}^{\tilde K_\alpha} \tilde a_k y_{\alpha, k}, \quad \tilde b=\sum_{k=0}^{\tilde K_\alpha} \tilde b_k y_{\alpha, k}
\end{equation}
such that
\begin{equation}
\label {modulation 2}
\left({Z}_{\omega_{\alpha, k}} -\omega^2_{\alpha, k} {W}_{\omega_{\alpha, k}}\right)(\tilde a_k, \tilde b_k)^T=({\tt h}_{1,k}, {\tt h_{2k}})
\end{equation}
where ${\tt h}_{j,k}={\tt h}_{j,k}(f, \bar v_1, v_2, w^+_1, w^-_1)$. The result of the above Lemma gives a map 
\[
(f, \bar v_1, v_2, w^+_1, w^-_1)\mapsto \tilde f
\]
and it is not hard (but somewhat tedious) to show that this map is Lipschitz with the Lipschitz constant less than $1$.

\subsubsection{Solution of the inner problem}

Having determined the modulation function $f$ as a function of the other unknowns  we are in position to solve the inner problem (\ref{fix 20}). We will rely on the theory of section \ref{lin theory} and in particular on Lemma \ref{lem 4}. Passing to interior variables we are to solve
\begin{equation}
\label{fix 36}
L_{\xi, \eta}{\phi}_2=-\left(\lambda\mu_\lambda\right)^{-2}\left[\rho{\tt N}_\lambda{\tt u}_0(\cdot;f)+{\tt E}_{int}\right]
\end{equation}
The right hand side of this equation is compactly supported in the set $|\eta|\leq m_{\bar\rho}M \log\beta$.  By Lemma \ref{lem fix 6}  it satisfies the hypothesis of Lemma \ref{lem 4} with $\theta$ and $\epsilon$ as in the hypothesis of Lemma \ref{lem fix 6} (we will restrict $\theta$ further when necessary). 

We will prove the following:
\begin{lemma}
\label{lem fix 7}
Given $\sigma\in (0, 1/4)$ let
\begin{align}
\label{chs t}
& 0<\theta<\min\left\{\bar\theta, 1-\frac{1}{m_1M}, \frac{1-\sigma}{4m_1M},\frac{\sigma}{m_\rho M}\right\}, \\
\label{chs e}
& 0<\epsilon<\min\left\{\frac{\sigma}{m_\rho M}-\theta,\frac{1}{8 m_{\bar \rho} M}, \frac{1-\sigma}{M m_1}-{4\theta}, \frac{\sigma}{4m_{\bar \rho} M}-\frac{\theta m_\rho}{4m_{\bar \rho}}, \frac{1-\sigma}{4m_{\bar \rho} M}-\frac{4\theta m_1}{4m_{\bar \rho}}\right\},
\end{align}
be fixed so that $\epsilon<\theta<\theta+\epsilon<\omega^*$. Suppose that the functions $(\bar v_1, v_2, w^+_1, w^-_1)$ satisfy the  hypothesis (\ref{fix 2})--(\ref{fix 3}). There exists  ${\phi}_2\in H^2_\theta(C_{\alpha|\gamma|})$ solving (\ref{fix 36}) and such that 
\begin{equation}
\label{fix 37}
\|{\phi}_2\|_{H^2_\theta(C_{\alpha|\gamma|})}<\beta^{-1/2+\sigma}. 
\end{equation}
\end{lemma}

\begin{proof}
Denote the right hand side of the equation (\ref{fix 36}) by ${\tt g}$ and write $\phi$ for $\phi_2$ for short.
We will  find a solution to 
\begin{equation}
\label{inner eq}
L_{\xi, \eta}\phi={\tt g}
\end{equation}
The function ${\tt g}$ is a nonlinear function of $(\bar v_1, v_2, w^+_1, w^-_1)$  only  since the modulation function has already been determined using  Lemma \ref{lem fix 6}. We will apply  Lemma \ref{lem 4}  
and taking the  projections  $P$ as in the cases (i)--(iv) of Lemma \ref{lem 4}  we will control 
$\|P\phi\|_{H^2_\theta(C_{\alpha|\gamma|})}$ by $\|P{\tt g}\|_{L^2_{\theta}(C_{\alpha|\gamma|})}$.

As in  (i) and (ii) in Lemma \ref{lem 4}, denoting by $P$ one of the projections $P_{>\theta+\epsilon}$ or $P_{\epsilon<\theta-\epsilon}$, we find
\begin{equation}
\label{fix 38}
\begin{aligned}
\|P{\tt g}\|_{L^2_{\theta}(C_{\alpha|\gamma|})}&\lesssim \beta^{-2}\|P\rho{\tt N}_\lambda{\tt u}_0(\cdot;f)\|_{L^2_{\theta}(C_{\alpha|\gamma|})}+
\beta^{-2}\|P{\tt E}_{int}\|_{L^2_\theta(C_{\alpha|\gamma|})}\\
&\lesssim \beta^{-2}\left(\|\rho{\tt N}_\lambda{\tt u}_0(\cdot;f)\|_{L^2_{\theta}(C_{\alpha|\gamma|})}+
\|{\tt E}_{int}\|_{L^2_\theta(C_{\alpha|\gamma|})}\right)\\
&\lesssim \beta^{-1/2+\theta m_\rho M}\log^4\beta+\beta^{-1+\sigma}\log^3\beta+\beta^{-3/2+4\theta m_1 M}\\
&< \frac{1}{3}\beta^{-1/2+\sigma} 
\end{aligned}
\end{equation}
where the next to the last estimate follows from (\ref{fix 15}) in Lemma \ref{lem fix 3} and (\ref{est int}) in Lemma \ref{lem fix 4} and the last estimate follows taking $\beta$ large enough.

In the case corresponding to (iii) we have to consider all terms involved in ${\tt g}$ since their supports vary somewhat. This can be done  following the proof of Lemma \ref{lem fix 4}. We get
\begin{equation}
\label{fix 39}
\begin{aligned}
\|P_{\theta-\epsilon<\theta+\epsilon}{\tt \phi}\|_{H^2_{\theta}(C_{\alpha|\gamma|})} &\lesssim \log\beta\left(\beta^{-1/2+(\theta+\epsilon)m_\rho M}+\beta^{-1+\sigma+\epsilon m_{\bar \rho} M}\log^3\beta+\beta^{-3/2+(\epsilon+4\theta)m_1 M}\right)\\
&<\frac{1}{3} \beta^{-1/2+\sigma},
\end{aligned}
\end{equation}
for $\beta$ large. 

In the case (iv) we choose $\tilde\theta$ such that $\theta+8\epsilon>\tilde \theta>\theta+4\epsilon$. We have
\[
\begin{aligned}
\|P_{<\epsilon} {\tt \phi}\|_{H^2_{\theta}(C_{\alpha|\gamma|})}\lesssim  \frac{1}{|\tilde \theta-\theta|}\|P_{<\epsilon}{\tt g}\|_{L^2_{\tilde \theta}(C_{\alpha|\gamma|})}&\lesssim \frac{1}{\epsilon}\beta^{(\tilde\theta-\theta)m_{\bar \rho}M}\|{\tt g}\|_{L^2_\theta(C_{\alpha|\gamma|})} \\
&\lesssim \frac{1}{\epsilon}\beta^{4\epsilon m_{\bar \rho}M}\|{\tt g}\|_{L^2_\theta(C_{\alpha|\gamma|})}.
\end{aligned}
\]
Comparing this with the next to the last line in (\ref{fix 38}) we get
\begin{equation}
\label{fix 40}
\begin{aligned}
\|P_{<\epsilon} {\tt \phi}\|_{H^2_{\theta}(C_{\alpha|\gamma|})}&\lesssim \beta^{4\epsilon m_{\bar\rho} M}\left(\beta^{-1/2+\theta m_\rho M}\log^4\beta+\beta^{-1+\sigma}\log^3\beta+\beta^{-3/2+4\theta m_1 M}\right)
\\
&<\frac{1}{3} \beta^{-1/2+\sigma}.
\end{aligned}
\end{equation}
using the hypothesis on $\epsilon$ and taking $\beta$ large. 
Combining all these estimates we find that 
\begin{equation}
\label{fix 41}
\|\phi\|_{H^2_{\theta}(C_{\alpha|\gamma|})}< \beta^{-1/2+\sigma}. 
\end{equation}
This ends the proof. 
\end{proof}

\subsubsection{Solution of the matching problem and  conclusion of the proof}

Here we rely on the theory developed in section \ref{general improvement}. Assuming that the modulation function $f$ is given we will solve the inner problem (using Lemma \ref{lem fix 7}) and  outer problem (in what follows)  for $(\bar v_1, v_2,w^+_1,w^-_1 )$.  To start with the latter  we suppose that the functions $(f,  v_2, \bar v_1, w^+_1, w^-_1)$ satisfying (\ref{fix 1})--(\ref{fix 3}) and  the matching condition (\ref{fix 22}) are given. We  look for functions $\bar\phi\in H^2(\mathcal K_\gamma)$, $\phi^\pm\in H^2(\Omega^\pm)$ such that  (see (\ref{fix 21})):
\begin{equation}
\label{fix 42}
\begin{aligned}
-\Delta \phi^\pm&=(1-\rho) N_\lambda (u_0) + E^\pm_{out},\qquad  \mbox{in}\ \Omega^\pm,\\
\phi^\pm&=0, \qquad \mbox{on}\ \partial\Omega^\pm\cap\partial\Omega,
\end{aligned}
\end{equation}
where assuming that 
\[
\bar \phi= \bar h_1^\pm+\bar h_2^\pm \eta+\mathcal O(e^{\,-a_0|\eta|}), \quad \eta\to \pm \infty,
\]
we will require 
\begin{equation}
\label{fix 43}
\begin{aligned}
\phi^\pm&=\bar h_1^\pm,\qquad \mbox{on}\ \gamma,\\
\partial_n \phi^\pm&={\lambda\mu_\lambda}\bar h_2^\pm,  \qquad \mbox{on}\ \gamma.
\end{aligned}
\end{equation}
Keep in mind that $\bar h^\pm_j=\bar h^\pm_j(s)$. We need to show that, consistently, the functions $\phi^\pm$  satisfy (\ref{fix 3}), the function $\bar\phi$ satisfies 
\[
\|\bar\phi\|_{H^2(\mathcal K_\gamma)}\leq \beta^{-1+\sigma}
\]
and that the matching condition (\ref{fix 22}) between these functions holds. Let us say at this point a few words about the choice of different constants. First, we choose $\theta$, $m_1<m_\rho$ so that (\ref{fix 32}) and (\ref{chs t}) hold. Next, we choose $m_2$ so that  (\ref{fix 32}) is satisfied and finally we chose $m_{\bar \rho}$ so that we can find $\epsilon>0$ in (\ref{chs e}), this may require further adjustment of $m_2>m_{\bar \rho}$.  Clearly, all these conditions can be satisfied by  an open set of parameters. 

We use  Proposition \ref{cru} to solve  (\ref{fix 42})--(\ref{fix 43}). By this proposition, Corollary  \ref{cor cru}, Lemma \ref{lem fix 3}, Lemma  \ref{lem fix 5} and Corollary \ref{cor err out} we have 
\begin{equation}
\begin{aligned}
\label{fix 44}
&\|\phi^+\|_{H^2(\Omega^+)}+\|\phi^-\|_{H^2(\Omega^-)}\leq \beta^{-1+\sigma}\\
&\|\bar\phi\|_{H^2(\mathcal K_\gamma)}\leq \beta^{-1+\sigma}
\end{aligned}
\end{equation}
Finally, we use Lemma \ref{lem patch} to check the matching condition between $\phi^\pm$ and $\bar\phi$ holds. As a consequence we have defined a map which to the given functions $(f, v_2, \bar v_1, w^+_1, w^-_1)$ assigns $\phi^\pm, \bar\phi$ in such a way that (\ref{fix 3}) and the matching condition are satisfied. 

We are in position to conclude the proof of the theorem. The equations defining the map $\mathfrak F$ 
\[
(f, v_2, \bar v_1, w^+_1, w^-_1)\longmapsto (\tilde f, \bar \phi_1,\phi,  \phi^+_1, \phi^-_1)
\]
are: the modulation equation (\ref{modulation 2a})--(\ref{modulation 2}), the inner equation (\ref{inner eq}) and the outer problem (\ref{fix 42})--(\ref{fix 43}). We have shown in Lemma \ref{lem fix 6}, Lemma \ref{lem fix 7} and Proposition \ref{cru} and its consequence  (\ref{fix 44}) that when the data satisfies (\ref{fix 1})--(\ref{fix 3}) and  the matching condition (\ref{fix 22}) so do the solutions of these problems. It suffices now to show that the map $\mathfrak F$ is a Lipschitz contraction.  To do this we write a system consisting of (\ref{modulation 2a})--(\ref{modulation 2}), the inner equation (\ref{inner eq}) and the outer problem (\ref{fix 42})--(\ref{fix 43}). It is standard to show that the right hand sides of these equations are indeed contractions and thus applying   the Banach contraction mapping  theorem we conclude that $\mathfrak F$ has a fixed point. This gives a solution of (\ref{liu 1}).  From the construction we obtain immediately (\ref{liu 5}). Similarly, using additionally (\ref{abplus})--(\ref{mu}) and (\ref{lzero}) we obtain (\ref{liu 6}) by a simple calculation (details can be found in \cite{kprv2018}). This completes the proof.

%\bibliography{biblio_29_08_2016}
%\bibliographystyle{amsplain}

\providecommand{\bysame}{\leavevmode\hbox to3em{\hrulefill}\thinspace}
\providecommand{\MR}{\relax\ifhmode\unskip\space\fi MR }
% \MRhref is called by the amsart/book/proc definition of \MR.
\providecommand{\MRhref}[2]{%
  \href{http://www.ams.org/mathscinet-getitem?mr=#1}{#2}
}
\providecommand{\href}[2]{#2}

\end{document}